\documentclass{article}

\usepackage{tikz-cd}
\usepackage{amsmath,amsthm,amssymb}
\usepackage{hyperref}
\usepackage{cite}

\newtheorem{theorem}{Theorem}[subsubsection]
\newtheorem{corollary}[theorem]{Corollary}%[subsection]
\newtheorem{proposition}[theorem]{Proposition}%[subsection]
\theoremstyle{definition}
\newtheorem{definition}[theorem]{Definition}
\newtheorem{notation}[theorem]{Notation}
\newtheorem{lemma}[theorem]{Lemma}
\theoremstyle{remark}
\newtheorem{remark}[theorem]{Remark}
\newtheorem{example}[theorem]{Example}

\newcommand{\ts}[2]{{_{#1}}\!\!\times_{#2}\!}
\newcommand{\anc}[0]{\mathcal {\varrho}}
\newcommand{\prolong}[0]{A\ts{\anc}{T(\pi)}T(A)}

\renewcommand{\sp}[0]{\hat{+}}
\newcommand{\sd}[0]{\div}

\title{Involution algebroids: a generalisation of Lie algebroids for tangent categories}
\author{Matthew Burke\footnote{Department of Mathematics and Statistics, University of Calgary, Calgary, Canada.}~~and Benjamin MacAdam\footnote{Department of Computer Science, University of Calgary, Calgary, Canada.}}

\begin{document}
\maketitle
\abstract{
	We define involution algebroids which generalise Lie algebroids to the abstract setting of tangent categories.
	As a part of this generalisation the Jacobi identity which appears in classical Lie theory is replaced by an identity similar to the Yang-Baxter equation.
	Every classical Lie algebroid has the structure of an involution algebroid and every involution algebroid in a tangent category admits a Lie bracket on the sections of its underlying bundle.
	As an illustrative application we take the first steps in developing the homotopy theory of involution algebroids.
}
\tableofcontents

\section{Introduction} % (fold)
\label{sec:introduction}

In this paper we generalise Lie algebroids (see \cite{MR2157566}) using a new class of algebraic structures called \emph{involution algebroids}.
A key algebraic component of a Lie algebroid is the Lie bracket on the sections of the underlying vector bundle of the algebroid.
By contrast the definition of involution algebroid does not refer to sections but instead asserts the existence of an involution on (a certain prolongation of) the total space of the bundle satisfying the axioms described in \ref{sub:definition_of_involution_algebroid}.
Under this reformulation the axiom that replaces the Jacobi identity of classical Lie theory is similar to the Yang-Baxter equation (see \ref{rem:yang-baxter}).

The definition of involution algebroid given in \ref{sub:definition_of_involution_algebroid} makes sense in any tangent category (see \cite{DIA_1984__12__A3_0} and \cite{MR3192082}) where the appropriate limits exist and are preserved by the tangent bundle functor.
In particular it does not involve function spaces or indeed rely on the theory of smooth functions at all.
This means that the morphisms of involution algebroids are defined in \ref{def:morphism-inv-algd} as the bundle maps commuting with the involution.
(By contrast the usual definition of morphisms between Lie algebroids requires the idea of `related sections' as described in \cite{MR1037400}.)
Furthermore the appropriate reformulation of the definition of \emph{admissible homotopy} in an involution algebroid in \ref{def:admissible-homotopy} avoids using an integral (or directly appealing to the existence of a connection) as in 1.3 of \cite{MR1973056}.

In classical Lie theory every Lie group has associated to it a Lie algebra (see 3.5 of \cite{MR2157566}) that can be thought of as a linear approximation to the Lie group.
Heuristically speaking the Lie bracket of this approximating Lie algebra encodes the commutator of the group multiplication.
(A similar intuition applies to integrable Lie algebroids.)
In \ref{subsec:the_involution_algebroid_of_a_groupoid} we show that every groupoid in an appropriately complete tangent category can be approximated by an involution algebroid.
The correct replacement for the group commutator turns out to be an operation similar to conjugation.
In \ref{subsub:the_involution_algebroid_of_a_lie_algebroid} we show that every Lie algebroid is an example of an involution algebroid.
In order to ease the calculations in this section we use the Levi-Civita connection as described in 3.1 of \cite{MR2914129} but the final results are independent of the choice of any connection.
The special case of Lie algebras is worked out in elementary terms in \ref{subsub:inv-alg-of-lie-alg} where the involution map is the endo-arrow on $A\times A\times A$ defined by $\sigma:(v, w_H, w_V) \mapsto (w_H, v, w_V + [v, w_H])$.
In the other direction in section \ref{subsubsec:the_lie_bracket_of_an_involution_algebroid} we describe how to define a Lie bracket on the set of sections of an involution algebroid in a tangent category.
Furthermore under the additional assumption that the tangent category has a \emph{unit object $R$} (see \ref{sec:representable-units-function-algebras}) we demonstrate that the Leibniz law (a part of the structure of a Lie algebroid) holds for this bracket also.

The construction of the Lie bracket on the sections of an involution algebroid in \ref{subsubsec:the_lie_bracket_of_an_involution_algebroid} and the definition of the involution algebroid structure on a Lie algebroid in \ref{subsub:the_involution_algebroid_of_a_lie_algebroid} use the same equation.
This implies an injection on objects (see \ref{ssubub:injection_on_objects}) from the category of Lie algebroids to the category of involution algebroids in the category of smooth manifolds.
In \ref{subsub:equivalence_of_categories} we show that the category of involution algebras (involution algebroids over the trivial base space) is equivalent to the category of Lie algebras.

Therefore a natural question arises: is the category of Lie algebroids a full subcategory of the category of involution algebroids?
Although we do not answer this question in this paper we take the first step in this direction by working out some of the homotopy theory of involution algebroids following the theory in \cite{MR1973056}.
The idea is that section 5.1 of \cite{MR1973056} gives an equivalence of categories
\begin{equation*}
\begin{tikzcd}
	LieAlgd \rar[yshift=0.3cm]{w} \rar[phantom]{\perp} & LocGpd \lar[yshift=-0.3cm]{alg}
\end{tikzcd}
\end{equation*}
where $LocGpd$ is the category of local Lie groupoids.
Here $w$ is the \emph{Weinstein local groupoid} construction which forms the quotient of a special kind of path (\emph{admissible} paths) by a special type of homotopy (\emph{admissible} homotopies).
Therefore one approach to understanding the morphisms in $LieAlgd$ is to understand the homotopy theory of Lie algebroids.
In \ref{sec:the_homotopy_theory_of_involution_algebroids} we define admissible paths and admissible homotopies in an involution algebroid, describe how to transport elements of $A$ along these paths and homotopies and work out two special cases that arise in the composite $alg\circ w$.
The logical next step would be to identify the conditions on involution algebroids (and indeed the tangent categories they live in) that allow us to take the quotient involved in the Weinstein groupoid construction.
We leave this as future work.

Late in the preparation of this paper the authors became aware of section 4 of \cite{MR2147171} and proposition 1 in \cite{MR2394515} which describe an involution similar to the one we are proposing here.
It is unclear to the authors of the present paper whether the two involutions are the same and in particular whether the one defined in \cite{MR2147171} satisfies our `flip' axiom (the Yang-Baxter style equation in \ref{rem:yang-baxter}).
Therefore it is possible that \cite{MR2147171} might provide an alternative to the work presented in \ref{subsub:the_involution_algebroid_of_a_lie_algebroid} which shows that all Lie algebroids are examples of involution algebroids.
We keep section \ref{subsub:the_involution_algebroid_of_a_lie_algebroid} unchanged because we need this explicit form for comparison to our work in \ref{subsubsec:the_lie_bracket_of_an_involution_algebroid} and also because we make no assumptions involving the existence of dual bundles and differentials that are required in 4.1 of \cite{MR2147171}.

% section introduction (end)

\section{Background on tangent categories} % (fold)
\label{sec:background_on_tangent_categories}

We formulate most of the ideas in this paper using the axiomatic system given by the theory of tangent categories which was introduced in \cite{DIA_1984__12__A3_0} and further developed in \cite{MR3192082}.
A tangent category is a category $\mathbb{X}$ equipped with an endofunctor $T$ that behaves in an analogous way to the tangent bundle endofunctor on the category of smooth manifolds.

As such this paper is a part of the body of work that reformulates various parts of differential geometry in the language of tangent categories.
For instance in \ref{sub:differential_bundles} we recall the definition of \emph{differential bundle} introduced in \cite{MR3792842} which is the appropriate generalisation of the definition of smooth vector bundle for tangent categories.
In addition in \ref{sub:curve_objects} we present a modified version of a \emph{curve object} (see section 5 of \cite{MR3684725}) that allows one to talk about the solutions to dynamical systems in a tangent category.

In this section we recall the definitions of tangent category, differential bundle and curve object in order to fix notation and to call attention to the results that we require in the sequel.
For more details and examples see \cite{MR3192082}.

\subsection{Additive bundles and tangent categories} % (fold)
\label{sub:additive_bundles_and_tangent_categories}

The tangent space at a point $m$ of a smooth manifold $M$ is regarded as a linear approximation to the region of the space $M$ that is close to $m$.
In the classical case these approximations are represented by a vector spaces which (as we smoothly vary the base point $m$) assemble to form a vector bundle.
In a tangent category we instead use the more general structure of an additive bundle. 
The following definition and remarks are contained in 2.1 and 2.2 of \cite{MR3192082}.

\begin{definition}[Additive bundle]
	If $M$ is an object of a category $\mathbb{X}$ then an \emph{additive bundle $q$ over $M$} is a commutative monoid in the slice category $\mathbb{X} / M$.
\end{definition}

\begin{remark}
	Included in this data is a total space $E$, a bundle projection $q:E \rightarrow M$, a zero section $\xi:M \rightarrow E$ and an addition $+_q:E\ts{q}{q}E \rightarrow E$.
\end{remark}

\begin{remark}
	A \emph{morphism $\phi:q \Rightarrow q'$ of additive bundles} is a pair of arrows $(\phi_1, \phi_0)$ making
	\begin{equation*}
	\begin{tikzcd}
		E \rar{\phi_1} \dar{q} & E' \dar{q'}\\
		M \rar{\phi_0} & M'
	\end{tikzcd}
	\end{equation*}
	commute preserving zero and addition in the fibres.
\end{remark}

If $f:M \rightarrow N$ is a smooth map between smooth manifolds then it lifts to a smooth map $T(f):T(M) \rightarrow T(N)$ between the tangent bundles.
Categorically speaking we can encode various properties of the derivative (such as linearity) in terms of natural transformations between various iterates and limits of the functor $T$.

\begin{definition}[Tangent category]
	A \emph{tangent category} is a category $\mathbb{X}$ equipped with an endofunctor $T$ on $\mathbb{X}$ and natural transformations
	\begin{align*}
		p&:T \Rightarrow id \tag{projection onto base}\\
		0&:id \Rightarrow T \tag{zero section}\\
		+&: T\ts{p}{p}T \Rightarrow T \tag{addition in tangent spaces}\\
		l&: T \Rightarrow T^2 \tag{vertical lift}\\
		c&: T^2 \Rightarrow T^2 \tag{canonical flip}
	\end{align*}
	where we assume that all pullback powers of $p$ (e.g. $T\ts{p}{p}T$ etc..) exist and:-
	\begin{itemize}
		\item all pullback powers of $p$ are preserved by $T$
		\item $0$ is a section of $p$
		\item $cc = id$ and $T(c)cT(c) = c T(c) c$
		\item $cl=l$, $T(l)l = ll$ and $cT(c)l = T(l) c$ 
		\item if $M\in \mathbb{X}$ then $+_M$ and $0_M$ makes each $p_M$ an additive bundle
		\item if $M\in \mathbb{X}$ then $(l_M, 0_M):p \Rightarrow T(p)$ is an additive bundle morphism
		\item if $M\in \mathbb{X}$ then $(c, id):T(p) \Rightarrow p$ is an additive bundle morphism
		\item the following diagram is an equaliser:
		\begin{equation*}
		\begin{tikzcd}
			T(M)\ts{p}{p}T(M) \rar[rightarrowtail]{\mu} & T^2(M) \rar[yshift=3pt]{T(p)} \rar[yshift=-3pt][swap]{0pp} & T(M)
		\end{tikzcd}
		\end{equation*}
		where $\mu(a, b) = 0b + l a$.
	\end{itemize}
\end{definition}

Examples of tangent categories include the category of smooth manifolds and the infinitesimally linear objects in a well-adapted model of synthetic differential geometry (see \cite{MR2244115}).
For more details see \cite{MR3192082}.

In this paper we assume that the tangent bundle functor $T: \mathbb{X} \rightarrow \mathbb{X}$ associated to a tangent category $\mathbb{X}$ preserves certain limits in $\mathbb{X}$.
Rather than assert $T$ preserves all limits (which does not hold in the category of smooth manifolds) we instead assert that $T$ preserves certain specific limits at various points in the text.
In fact even the general assumption that $T$ preserves all limits is not an unreasonable one due to the main result in \cite{MR3725887} which shows that every tangent category embeds into another tangent category for which $T$ is representable and hence preserves all limits.
We also assume that we work in a tangent category `with negatives' which means that the additive bundles are all commutative groups in the appropriate slice category.

% subsection additive_bundles_and_tangent_categories (end)

\subsection{Differential bundles} % (fold)
\label{sub:differential_bundles}

Recall that every Lie algebroid can be obtained by placing extra structure on a smooth vector bundle.
Accordingly it turns out that in order to define involution algebroids in a tangent category we first need to understand \emph{differential bundles} which are the  appropriate generalisation of smooth vector bundles to this setting.
The following is definition 2.3 of \cite{MR3792842}.

\begin{definition}[Differential bundle]
	A \emph{differential bundle} is an additive bundle $q:E \rightarrow M$ equipped with a \emph{lift map} $\lambda: E \rightarrow TE$ such that:
	\begin{itemize}
		\item $(\lambda, 0)$ is an additive bundle morphism
		\item $(\lambda, \xi)$ is an additive bundle morphism
		\item $T(\lambda)\lambda = l\lambda$
		\item the following diagram is an equaliser:
		\begin{equation*}
		\begin{tikzcd}
			E\ts{q}{q}E \rar[rightarrowtail]{\mu} & T(E) \rar[yshift=3pt]{T(q)} \rar[yshift=-3pt][swap]{0qp} & T(M)
		\end{tikzcd}
		\end{equation*}
		where $\mu(a, b) = 0a +_q \lambda b$
	\end{itemize}
\end{definition}

\begin{example}
	If $q:E \rightarrow M$ is a smooth vector bundle and $m\in M$ then there is an isomorphism $\psi: E_m\times E_m \rightarrow T(E_m)$ because $E_m\cong \mathbb{R}^n$ for some $n\in \mathbb{N}$.
	In this case we can define $\lambda(e) = \psi(0_{qe}, e)$.
	In addition the map $\mu(a, b) = \psi(b, 0_m)+_q\psi(0_m, a) \cong \psi(b, a)$ is the lift defined on page 55 of \cite{MR1202431} where $m = qa = qb$.
\end{example}

\begin{definition}[Morphism of differential bundles]
	A \emph{morphism $\phi: q \Rightarrow q'$ of differential bundles} is a pair $(\phi_1, \phi_0)$ of arrows making
	\begin{equation*}
	\begin{tikzcd}
		E \rar{\phi_1} \dar{q} & E' \dar{q'}\\
		M \rar{\phi_0} & M'
	\end{tikzcd}
	\end{equation*}
	commute.
	Furthermore a \emph{linear differential bundle morphism} is a differential bundle morphism that preserves the lift: $T(\phi_1)\lambda = \lambda' \phi_1$.
\end{definition}

Next we recall two results about differential bundles that we require in the sequel.

\begin{lemma}[Pullbacks of differential bundles]\label{lem:pullbacks-in-caty-of-diff-bundles}
	Let 
	$$(q_0, +_0, \xi_0, \lambda_0) \text{ , }(q_1, +_1, \xi_1, \lambda_1)\text{ and }(q_2, +_2, \xi_2, \lambda_2)$$
	be differential bundles with total spaces $E_0$, $E_1$ and $E_2$ and base spaces $M_0$, $M_1$ and $M_2$ respectively.
	If $\phi: q_1 \Rightarrow q_0$ and $\psi: q_2 \Rightarrow q_0$ are linear differential bundle morphisms then
	\begin{equation*}
	q_1\ts{\phi}{\psi}q_2:=(q_1\times q_2, +_1\times +_2, \xi_1\times \xi_2, \lambda_1\times \lambda_2)
	\end{equation*}
	is a differential bundle with total space $E_1\ts{\phi_1}{\psi_1}E_2$ and base space $M_1\ts{\phi_0}{\psi_0}M_2$.
	Moreover $q_1\ts{\phi}{\psi}q_2$ is the pullback of $\phi$ and $\psi$ in the category of differential bundles.
	\begin{proof}(Sketch.)
		First recall that we assume all of the limits involved exist and are preserved by $T$.
		Second recall 2.16 of \cite{MR3792842} which states that if $\phi$ is a linear morphism of differential bundles then $\phi$ automatically preserves the addition and zero.
		It is then a lengthy but straightforward calculation to check that the differential bundle axioms hold.
		As a representative example the addition is unital because:
		\begin{align*}
			\xi q + id &= (\xi_1q_1\times \xi_2q_2)+(id\times id)\\
			&= (\xi_1q_1+_1 id)\times (\xi_2 q_2+_2 id)\\
			&= id\times id = id
		\end{align*}
		where we have written $+$ for $+_1\times +_2$, $\xi = \xi_1\times \xi_2$ and $q = q_1\times q_2$.
		Note also that 
		\begin{equation*}
		\begin{tikzcd}[column sep = 0.4cm]
			(E_1 \ts{\phi_1}{\psi_1} E_2)\ts{q}{q}(E_1\ts{\phi_1}{\psi_1}E_2)  \rar[rightarrowtail]{\mu} & T(E_1\ts{\phi_1}{\psi_1}E_2) \rar[yshift=3pt]{T(q)} \rar[yshift=-3pt][swap]{0qp} & T(M_1\ts{\phi_0}{\psi_0}M_2)
		\end{tikzcd}
		\end{equation*}
		is an equaliser because limits commute with limits.
	\end{proof}
\end{lemma}

The following lemma is 2.5 in \cite{MR3792842}.

\begin{lemma}\label{lem:T-of-diff-bundle}
  If $(E, q, +_q, \xi, \lambda)$ is a differential bundle then 
  $$(T(E), T(q), T(+_q), T(\xi), cT(\lambda))$$
  is a differential bundle.
\end{lemma}

Recall that a vector bundle over the singleton base space is a vector space.
As in \cite{MR3792842} we define a \emph{differential object} to be a differential bundle over the terminal object.
The following proposition (3.4 in \cite{MR3792842}) gives a more elementary characterisation of differential objects.

\begin{proposition}\label{def:differential-object}
  The following are equivalent:
  \begin{enumerate}
  \item $(E, \oplus_E, \xi_E, \lambda_E)$ is a differential bundle over the terminal object.
  \item $(E, \oplus_E, \xi_E)$ is a commutative monoid object such that.
    \begin{enumerate}
    \item $TE$ satisfies the following biproduct diagram in $\mathsf{CMon}(\mathbb{X})$
      \[
        \begin{tikzcd}
          E \ar[rr, equals] \ar[rd, "\lambda"] && E \\
          & T(E) \ar[ru, "\hat{p}"] \ar[rd, "p"] \\
          E \ar[rr, equals] \ar[ru, "0"] && E
        \end{tikzcd}
      \]
    \item The two additions are compatible:
      \[
        \begin{tikzcd}
          E \ar[r, "!_E"] \ar[d, "0_E"'] & 1 \ar[d, "\xi"] \\
          T(E) \ar[r, "\hat{p}"'] & A
        \end{tikzcd}
        \begin{tikzcd}
          T_2E \ar[d, "+_p"] \ar[r, "( \hat{p}\pi_0 {,} \hat{p}\pi_1 )"] & E \times E \ar[d, "+_A"] \\
          T(A) \ar[r, "\hat{p}"'] & E
        \end{tikzcd}
      \]
    \item The biproduct structure is coherent with the lift $l$:
      \[
        \begin{tikzcd}
          TE \ar[d, "\hat{p}"'] \ar[r, "l"] & T^2E \ar[d, "\hat{p}"]\\
          E & TE \ar[l, "\hat{p}"]
        \end{tikzcd}
      \]
    \end{enumerate}
  \end{enumerate}
\end{proposition}

% subsection differential_bundles (end)

\subsection{Affine bundles} % (fold)
\label{sub:affine_structures}

The definition of the Lie bracket in a tangent category may be simplified using the \emph{affine structure}
of the second tangent bundle. If $V$ is a differential object, we say that $A$ is
\emph{affine} over $V$ if there is an action of $V$ on $A$ 
\[
  \sp: A \times V \to A
\]
which is linear in $V$, so that $(v + v') \sp a = v \sp (v' \sp a)$.
There is also a \emph{strong difference}
\[
  \sd: A \times A \to V
\]
so that $a \sp (a \sd a') = a$. We can find a similar structure on the tangent
bundle of a differential bundle and the following proposition lays out the algebraic
necessities to find this structure.
For a textbook treatment of these ideas in the context of synthetic differential geometry see V.4 of \cite{MR1083355}.

\begin{proposition}
  Let $\pi: A \rightarrow M$ be a differential bundle. 
  There are two additive bundle structures on $T(A)$, $+_{T\pi} := T(+_\pi)$ and $+_p$ that satisfy the following identities:
  \begin{enumerate}
  \item Interchange: $(x +_{T\pi} y) +_p (w +_{T\pi} z) = (x +_p w) +_{T\pi} (y +_p z)$.
  \item If $v:A$, then
    \[(\lambda(v) +_{T\pi} 0p(a)) +_p a = (\lambda(v) +_p \xi\pi(a)) +_{T\pi} a\]
  \item If $p(x) = p(y), T(\pi)(x) = T(\pi)(y)$, then
    \[(x -_A y) -_{T\pi} 0py = (x -_{T\pi} y) -_A \xi\pi y\]
  \item
    \begin{itemize}
    \item $T(\pi)((x -_A y) -_{T\pi} 0py = 0\pi p y$
    \item $p((x -_{T\pi} y) -_A T(\xi)T(\pi)y) = \xi p T(\pi)y$
    \end{itemize}
  \item $p((x -_A y) -_{T\pi} 0py) = 0ppx, T(\pi)(x -_A y) -_{T\pi} 0py = T(\pi)(x)$
  \end{enumerate}
\end{proposition}
\begin{proof}
\begin{itemize}
\item[(1)] Lemma 2.8 in \cite{MR3792842}.
\item[(2)] Calculate:
  \begin{align*}
    (\lambda(v) +_{T\pi} 0pa) +_p a
    &= (\lambda(v) +_{T\pi} 0pa) +_p (T(\xi)T(\pi)a +_{T\pi} a) \\
    &= (\lambda(v) +_p T(\xi)T(\pi)a) +_{T\pi} (0pa +_p a)
    & \tag{interchange}\\
    &= (\lambda(v) +_p T(\xi)T(\pi)a) +_{T\pi} a 
  \end{align*}
\item[(3)]
  Calculate:
  \begin{align*}
    (x -_A y) -_{T\pi} 0py
    &= (x +_p -_Ay) +_{T\pi} -_{T\pi}0py\\
    &= (x +_p -_Ay) +_{T\pi} (-_{T\pi}(y) +_p -_A-_{T\pi}(y)) \\
    &= (x +_{T\pi} -_{T\pi}(y)) +_p (-_Ay +_{T\pi} -_A-_{T\pi}(y)) \\
    &= (x +_{T\pi} -_{T\pi}(y)) +_p (-_Ay +_{T\pi} -_{T\pi}-_Ay) \\
    &= (x -_{T\pi} y) -_A T(\xi)T(\pi)y
  \end{align*}
\item[(4)]
  Observe that:
  \begin{align*}
    T(\pi)((x -_A y) -_{T\pi} 0py)
    &= T(\pi)0py \\
    &= 0\pi p y
  \end{align*}
  and
  \begin{align*}
    p((x -_{T\pi} y) -_A T(\xi)T(\pi)y)
    &= pT(\xi)T(\pi)y \\
    &= \xi p T(\pi)y
  \end{align*}
\end{itemize}
\end{proof}

\begin{definition}[Affine bundle]
  Let $\pi: A \to M$ be a differential bundle. We say that
  $q: B \to Q$ is \emph{affine} over $\pi$ if there are maps:
  \begin{itemize}
  \item Strong difference: $\sd : B \ts{q}{q} B \to A$.
  \item Strong Sum: $\sp : B \times A \to B$
  \end{itemize}
  so that:
  \begin{itemize}
  \item Associativity: if $\pi(v) = \pi(w)$ then
    $(a \sp v) \sp w = a \sp (v +_\pi w)$
  \item Inverse: if $q(a) = q(b)$ then $a \sp (a \sd b) = a$
  \end{itemize}
  where $v, w\in A$ and $a, b\in B$.
\end{definition}

\begin{example}[The tangent bundle of a differential bundle]
  Consider the tangent bundle of a differential bundle $\pi: A \to M$.
  By the earlier lemma, we may define the bundle:
  \[
    q: T(A) \to A \ts{\pi}{p} T(M)
  \]
  and we have an affine structure induced by:
  \begin{itemize}
  \item $a \sp v :=  (\lambda(v) +_{T\pi} 0pa) +_{Tp} a$
  \item $a \sd b :=
    \{
      (a -_A b) -_{T\pi} 0pb
    \}$
  \end{itemize}
\end{example}
\begin{remark}
  Consider the affine structure on the second tangent bundle of
  some object $M$. 
  The Lie bracket $[X,Y]_M$ is $cT(Y)X\sd T(X)Y$.
\end{remark}

% subsection affine_bundles (end)

% subsection representable_units
\subsection{Units and function algebras}
\label{sec:representable-units-function-algebras}

In the definition of Lie algebroid the function algebra $C^\infty(M, \mathbb{R})$ is used to formulate the Leibniz law (see \ref{subsub:the_involution_algebroid_of_a_lie_algebroid}).
However in a general tangent category there need not be an object that can appropriately play the role of $\mathbb{R}$ and so we introduce a special type of tangent category for which this object does exist.
Therefore in this section we describe a \emph{unit object $R$} that satisfies various universal properties that make it a good surrogate for the real line.
In fact we only give a sketch of the theory of unit objects relevant for this paper: the general theory will be further explored in \cite{GLM}.

We use a unit object in two places in this paper.
First in \ref{subsubsec:the_lie_bracket_of_an_involution_algebroid} we prove that in a tangent category with unit object every involution algebroid has a Lie bracket on its set of sections that satisfies the Leibniz law.
Second in \ref{sub:transport_along_a_paths} we use a unit object that is also a curve object (see \ref{sub:curve_objects}) to construct a line bundle that is useful in the homotopy theory of Lie algebroids.
Now we give the definition of unit object in two stages.
In the sequel we use the term `unit object' for the second type of unit object (i.e. a `fibred' unit object).

\begin{definition}
  A \emph{(non-fibred) unit object} in a tangent category is a differential object $R$ with a point $u: 1 \to R$ such that for every morphism $f:V\times W \rightarrow E$ of differential objects there exists a unique lift
  \[
    \begin{tikzcd}
      V \times W \ar[r,"f"] \dar[swap]{(id, u!)} & E \\
      R \times (V \times W) \ar[ru, swap, dashed, "\hat{f}"]
    \end{tikzcd}
  \]
	such that $\hat{f}$ is linear in $R$. 
	If further $f$ is linear is $W$ then $\hat{f}$ is also.
\end{definition}

\begin{notation}[Scalar multiplication]
	If $M$ is a differential object then the unique map $\hat{id}:R \times M \rightarrow M$ lifting $id_M$ is denoted by the infix operator $\bullet_M$.
	In this notation: $\hat{f}(r, v, w) = r\bullet_M f(v, w)$.
\end{notation}

Before giving the definition of (fibred) unit object we recall two facts.
First 2.5 of \cite{MR3192082} tells us that every slice category of a tangent category is a tangent category.
Second 5.12 of \cite{MR3792842} tells us that differential bundles over $M$ are differential objects in the slice tangent category over $M$.

\begin{definition}[Unit object]\label{def:unit-object}
	A \emph{(fibred) unit object} is a differential object $R$ with a point $u:1 \rightarrow R$ such that:-
  \begin{itemize}
  \item The trivial bundle $\pi_1: R \times M \to M$ is a unit in the slice tangent category over $M$.
  \item Multiplication is preserved by substitution functors: if $\pi:A \rightarrow M$ is a differential bundle, $f:N \rightarrow M$ and $f^*(A)$ is the pullback of $\pi$ along $f$ then
    \[
      \begin{tikzcd}
        (R\times M, \pi_0) \times (f^*A, f^*\pi) \rar{id\times h_f} \arrow{d}[swap]{\bullet_{(f^*(A), f^*(\pi))}}
        & (R\times M, \pi_0) \times (A,\pi) \dar{\bullet_{(A, \pi)}} \\
        (f^*A, f^*\pi) \rar{h_f} & (A,\pi)
      \end{tikzcd}
    \]
    commutes where $h_f$ is the natural bundle map $f^*(A) \Rightarrow A$.
  \item Multiplication is torsion-free: $(\lambda_R r \bullet^T_\pi v) = r \bullet_{T\pi} v +_p T(\xi\pi)v$.
  \end{itemize}
\end{definition}

\begin{notation}[Scalar multiplication in fibres]
	If $\pi:A \rightarrow M$ is a differential bundle then we write:-
	\begin{itemize}
		\item $\bullet_{\pi}:R\times A \rightarrow A$ for $\bullet_{(A, \pi)}$
		\item $\bullet_p:R\times T(A) \rightarrow T(A)$ for $\bullet_{(T(A), p)}$
		\item $\bullet_{T(\pi)}:R \times T(A) \rightarrow T(A)$ for $\bullet_{(T(A), T(\pi))}$
		\item $\bullet_{\pi}^T:T(R)\times T(A) \rightarrow T(A)$ for $T(\bullet_{(A, \pi)})$
	\end{itemize}
	and note that $\bullet_p: R\times T(M) \rightarrow T(M)$ makes sense for any object $M$.
\end{notation}

Now we state without proof some consequences of \ref{def:unit-object} and again refer to \cite{GLM} for more details.
In a tangent category (with negatives) with a unit object $R$:
  \begin{itemize}
  \item The object $R$ is a commutative ring (if the tangent category does not
    negatives, this will be a commutative rig).
  \item The category of differential objects and linear maps is a full subcategory
    of $R$-modules: every differential object has a canonical $R$-module
    structure and a morphism is linear if and only if it is an $R$-module morphism.
  \item We may rewrite $\lambda_E(e) = (\lambda(u) \bullet^T 0e)$.
  \item When using local coordinates induced by a connection $(K, H)$ (see \ref{prop:decomp-full-connection}) the derivative of scalar multiplication has the form
\[
  (r_1,r_2) \bullet^T_\pi (w_H, v, w_V) = (r_1\bullet_{\pi}w_H, v, r_1\bullet_\pi w_V + r_1 \bullet_\pi w_H)
\]
which may be written without coordinates as:
\begin{equation}\label{eq:derivative-of-scalar-multiplication}
r\bullet^T_{\pi}w -_p pr\bullet_{T(\pi)} w = \hat{p}r\bullet_{T(\pi)}\lambda p w -_p T(\xi) 0m
\end{equation}
where $m = \pi p w$ and $\hat{p}$ is from \ref{def:differential-object}.
  \end{itemize}

We complete this section by describing a version of the $C^\infty$
functor for tangent categories which associates to each manifold its algebra of smooth functions.
\begin{remark}
  For every object $M$ of tangent category with a unit object $R$ there is a ring $C^\infty(M, R)$ defined to be $\mathbb{X}(M,R)$ with pointwise multiplication and addition. 
  This determines a functor $C^\infty: \mathbb{X} \to \mathsf{R-Alg}$.
  In addition the \emph{derivations} (the $R$-module morphisms $d: C^\infty(M) \to C^\infty(M)$ satisfying $d(f\cdot g) = df \cdot g + f \cdot dg$) form a Lie algebra with bracket given by $[d,d'] = dd' - d'd$.
\end{remark}

% subsection units (end)

\subsection{Curve objects} % (fold)
\label{sub:curve_objects}

In this section we describe the axiomatic system we use to perform integration in a tangent category.
The approach to integration we choose to generalise involves finding integral curves (solutions) of vector fields.
As such we follow the presentation in section 5 of \cite{MR3684725} but make some modifications concerning the solutions of linear vector fields.
% First recall that a vector field on a manifold $M$ assigns to every point $m\in M$ a tangent vector at $m$.
% This assignment is assumed to smoothly depend on $m$ and so a vector field is equivalently characterised as a smooth section $X: M \rightarrowtail T(M)$ of the tangent bundle $p:T(M) \twoheadrightarrow M$.
% Intuitively we think of a solution to a vector field as the locus traced out when we place an object at some initial location $x_0$ on the manifold and follow it as it moves in the direction specified by the vector field.
The following definition is 5.15 in \cite{MR3684725}.

\begin{definition}[Dynamical system]
	A \emph{dynamical system $(x_0, X)$ on $M$} consists of an initial condition $x_0: 1 \rightarrow M$ and a section $X:M \rightarrow T(M)$ of the tangent bundle $p:T(M) \rightarrow M$.
\end{definition}

\begin{definition}[Morphism of dynamical systems]
	If $\mathbb{X} = (x_0, X)$ and $\mathbb{Y} = (y_0, Y)$ are dynamical systems on $M$ and $N$ respectively then a \emph{morphism $f: \mathbb{X} \Rightarrow \mathbb{Y}$ of dynamical systems} is an arrow $f:M \rightarrow N$ making
	\begin{equation*}
	\begin{tikzcd}
		1 \rar{x_0} \drar[swap]{y_0} & M \dar{f} \rar{X} & T(M) \dar{T(f)}\\
		{} & N \rar{Y} & T(N)
	\end{tikzcd}
	\end{equation*}
	commute.
\end{definition}

Recall from 2.4 (iii) in \cite{MR3792842} that not every differential bundle in the category of smooth manifolds is a vector bundle.
The key additional property that we require in this section is the following `local triviality' condition where we have in mind the case $R = \mathbb{R}$ in the category of smooth manifolds.

\begin{definition}\label{def:locally-trivial-diff-bundle}
	A differential bundle $q:E \rightarrow M$ is \emph{locally trivial with respect to $R$} iff there exists $n\in \mathbb{N}$ such that for all $f:R \rightarrow M$ the following square 
	\begin{equation*}
	\begin{tikzcd}
		R\times R^n \rar{} \dar{\pi_0} & E \dar{q}\\
		R \rar{f} & M
	\end{tikzcd}
	\end{equation*}
	is a pullback.
\end{definition}

The following is definition 5.18 in \cite{MR3684725}.

\begin{definition}\label{def:linear-vector-field}
	If $q: E \rightarrow M$ is a differential bundle then a \emph{linear vector field $(X^E, X^M)$} is a linear morphism of differential bundles
	\begin{equation*}
	\begin{tikzcd}
		E \rar{X^E} \dar{q} & T(E) \dar{T(q)}\\
		M \rar{X^M} & T(M)
	\end{tikzcd}
	\end{equation*}
	such that $X^E$ and $X^M$ are vector fields.
	A dynamical system $(x_0, X^E)$ is \emph{linear over another dynamical system $(m_0, X^M)$} iff $(X^E, X^M)$ is a linear vector field and $q(x_0) = m_0$.
\end{definition}

Now we present our modification of definition 5.19 in \cite{MR3684725}.

\begin{definition}[Complete curve object]\label{def:complete-curve-object}
	A \emph{complete curve object $\mathbb{R}$} in a tangent category $\mathbb{X}$ is a dynamical system 
	\begin{equation*}
	\mathbb{R} = (0_R:1 \rightarrow R, \partial: R \rightarrow T(R))
	\end{equation*}
	such that:-
	\begin{itemize}
		\item there is another point $1_R:1 \rightarrow R$
		\item if $f, g: \mathbb{R} \Rightarrow \mathbb{X}$ are morphisms of dynamical systems then $f = g$
		\item let $q:E \rightarrow M$ be a locally trivial differential bundle with respect to $R$.
		If $(x_0, X^E)$ is linear over $(q(x_0), X^M)$ and $(q(x_0), X^M)$ has a complete solution then $(x_0, X^E)$ has a complete solution.
	\end{itemize}
\end{definition}

\begin{remark}
	The requirement in \ref{def:complete-curve-object} that $R$ is bipointed is not actually used in this paper.
	However we include it here because we look forward to developing Lie's second and third theorems in a tangent category where it will be required to define the target map of the groupoid integrating an involution algebroid.
\end{remark}

We now show that the manifold $\mathbb{R}$ is a complete curve object in the tangent category of smooth manifolds.
So for the rest of this section all our objects are smooth manifolds and all our arrows are smooth maps.

\begin{example}[Unit dynamical system]
	If $S$ is an interval in $\mathbb{R}$ containing $0$ then we denote by $\partial:S \rightarrow T(S)$ the vector field $\partial:x \mapsto (x, 1)$.
	The \emph{unit dynamical system $\mathbb{U}S$} is the dynamical system $(0, \partial)$.
\end{example}

\begin{definition}[Local solutions]
	If $(x_0, X)$ is a dynamical system then a \emph{local solution $\gamma$ to $(x_0, X)$} is a morphism $\gamma:\mathbb{U}S \Rightarrow \mathbb{X}$ of dynamical systems for some interval $S$ containing $0$.
\end{definition}

The following lemma corresponds to the key classical result concerning the existence and uniqueness of solutions to differential equations.
We state it without proof.

\begin{lemma}[Maximal solution]
	If $(x_0, X)$ is a dynamical system on $M$ then there exists an open interval $(a, b)$ of $\mathbb{R}$ containing $0$ and an arrow $\gamma:(a, b) \rightarrow M$ such that:-
	\begin{itemize}
		\item $\gamma$ is a solution to $(x_0, X)$
		\item if $S$ is an interval in $\mathbb{R}$ strictly containing $(a, b)$ then there is no solution $S \rightarrow M$ to $(x_0, X)$
		\item if $\gamma'$ is another solution to $(x_0, X)$ with domain $(a, b)$ then $\gamma = \gamma'$.
	\end{itemize}
\end{lemma}

\begin{definition}[Completeness]
	A \emph{complete solution to a dynamical system $(x_0, X)$} is a solution with domain $\mathbb{R}$.
	A \emph{complete vector field $X$} is a vector field such that for all initial conditions $x_0$ there exists a complete solution to $(x_0, X)$.
\end{definition}

\begin{lemma}[Linear differential equations]
	If $(x_0,X)$ is a dynamical system on $\mathbb{R}^n$ for some $n\in \mathbb{N}$ such that $\pi_1 X: \mathbb{R}^n \rightarrow \mathbb{R}^n$ is a linear then $X$ is complete.
	\begin{proof}
		This is a classical result: use $\gamma(t) = (\exp(t\pi_1 X)){x_0}$ where $\exp$ denotes the matrix exponential.
	\end{proof}
\end{lemma}

We split the proof of the lifting of complete solutions into two steps.
The first concerns a vector field on a Euclidean space.

\begin{lemma}\label{lem:sub-linear-vector-field}
	If $X: \mathbb{R}\times \mathbb{R}^n \rightarrow T(\mathbb{R}\times \mathbb{R}^n)$ is the vector field defined by
	\begin{equation*}
	X:(x, \vec{v}) \mapsto (x, \vec{v}, 1, A_x(\vec{v}))
	\end{equation*}
	where $A_x(\vec{v})$ is smooth in $x$ and linear in $\vec{v}$ then $X$ is complete.
	\begin{proof}
		Let $\gamma:(a, b) \rightarrow \mathbb{R}\times \mathbb{R}^n$ be the maximal solution to $(x_0, X)$ for some initial condition $x_0\in \mathbb{R}\times \mathbb{R}^n$.
		Suppose (for the purpose of obtaining a contradiction) that $b<\infty$.
		Let $M = sup_{x\in [a,b]}\|A_x\|$ where $\|-\|$ denotes the operator norm.
		Then the vector field
		\begin{equation*}
		\hat{X}:(x, \vec{v}) \mapsto (x, \vec{v}, 1, M\vec{v})
		\end{equation*}
		is linear and so has a complete solution $\hat{\gamma}:\mathbb{R} \rightarrow \mathbb{R}\times \mathbb{R}^n$.
		Now by construction $|\gamma(y)|\leq |\hat{\gamma}(y)|$ for all $y\in [0, b)$.
		Therefore $\gamma[0, b)$ is bounded because $\hat{\gamma}[0,b]$ is bounded.
		But now we can extend $\gamma$ smoothly to $[0, b]$: 
		\begin{equation*}
		\gamma(b):= \lim_{n \rightarrow \infty} \gamma\left(b-\frac1n\right)
		\end{equation*}
		which exists because $\gamma[0, b)$ is bounded.
		This extension of $\gamma$ is still a solution to $(x_0, X)$ because the derivative of $\gamma$ is smooth.
		However this means that $(a, b)$ wasn't the domain for the maximal solution to $(x_0, X)$ and so $b=\infty$.
		Similarly $a = -\infty$.
	\end{proof}
\end{lemma}

The second step is to use the complete solution of $X^M$ to reduce the general case to a vector field on a Euclidean space of the form described in the statement of \ref{lem:sub-linear-vector-field}.

\begin{proposition}
	Let $q: E \rightarrow M$ be a locally trivial differential bundle with respect to $\mathbb{R}$ and let $X^E$ be a linear vector field over $X^M$.
	If for all $x_0\in E$ the vector field $(q(x_0), X^M)$ has a complete solution then $(x_0, X^E)$ has a complete solution.
	\begin{proof}
		If $f:\mathbb{R} \rightarrow M$ is the solution to $X^M$ then we can restrict $X^E$ to a section $X:\mathbb{R}\times \mathbb{R}^n \rightarrow T(\mathbb{R}\times \mathbb{R}^n)$ by pulling back along $f$:
		\begin{equation*}
		\begin{tikzcd}
			T(\mathbb{R}\times \mathbb{R}^n) \arrow{rrr}{T(p_1)} \arrow{ddd}{T(\pi_0)} & {} & {} & T(E) \arrow{ddd}{T(q)} \\
			{} & \mathbb{R}\times \mathbb{R}^n \dar{\pi_0} \ular[dashed]{X} \rar{p_1} & E \dar{q} \urar{X^E} & {}\\
			{} & \mathbb{R} \rar{f} \dlar{\partial} & M \drar{X^M} & {} \\
			T(\mathbb{R}) \arrow{rrr}{T(f)} & {} & {} & T(M)
		\end{tikzcd}
		\end{equation*}
		where the middle (and outer) squares are pullbacks because $\mathbb{R}$ is contractible.
		Now $X$ is given in coordinates by:
		\begin{equation*}
		X:(x, \vec{v}) \mapsto \left(x, \vec{v}, 1, X^E_{f(x)}(\vec{v})\right)
		\end{equation*}
		and so the result follows from \ref{lem:sub-linear-vector-field}.
	\end{proof}
\end{proposition}

% subsection linear_vector_fields (end)

% subsection curve_objects (end)
% \changelocaltocdepth{2}
% section background_on_tangent_categories (end)

\section{Involution algebroids} % (fold)
\label{sec:involution_algebroids_and_examples}

If $X$ and $Y$ are two smooth vector fields on a smooth manifold $M$ then their Lie bracket can be expressed as an algebra-theoretic commutator (see lemma 4.12 in \cite{MR1930091}) or alternatively as a group theoretic commutator (see section I.9 of \cite{MR2244115}).
In the same way the Lie bracket on a Lie algebra (or integrable Lie algebroid) may be obtained as a commutator of infinitesimally small elements of the Lie group (or groupoid) that integrates the algebra.
For the definition of an involution algebroid we instead axiomatise a structure that corresponds to operation similar to conjugation on infinitesimal elements of a groupoid.
In this way we replace the Lie bracket of a Lie algebroid (which acts on sections) with an involution which acts on a certain prolongation of the algebroid.
Under this replacement the Jacobi identity of classical Lie theory is replaced by an identity that is similar to the Yang-Baxter equation.

In section \ref{sub:definition_of_involution_algebroid} we define involution algebroids in the abstract setting of a tangent category.
Then we move on to giving classes of examples of involution algebroids.
In section \ref{subsec:the_involution_algebroid_of_a_groupoid} we show how to obtain a involution algebroid as an linear approximation of a groupoid in a tangent category.

\subsection{Anchored bundles and prolongations} % (fold)
\label{sub:anchored_bundles_and_prolongations}

Every Lie algebroid $\pi:A \rightarrow M$ is obtained by placing additional structure on a smooth vector bundle.
One part of this additional structure is a vector bundle map $\anc:A \Rightarrow T(M)$ called the \emph{anchor}.
The anchor map specifies how $A$ behaves like a generalised tangent bundle: an element $a\in A$ is related to the element $\anc a$ in an analogous way to how an arrow $g$ in a groupoid $s, t:G \rightrightarrows M$ is related to the pair $(sg, tg)$.
Since a differential bundle is the appropriate analogue of a smooth vector bundle in a tangent category we now define anchored bundles in terms of differential bundles.

\begin{definition}[Anchored bundle]
	An \emph{anchored bundle} is a differential bundle $\pi:A \rightarrow M$ equipped with a linear bundle morphism $\anc: A \Rightarrow TM$.
\end{definition}

\begin{definition}[Morphism of anchored bundles]
	If $A$ and $B$ are anchored bundles with anchors $\anc_A$ and $\anc_B$ respectively then \emph{a morphism $f:A \rightarrow B$ of anchored bundles} is a morphism $(f_0, f_1): A \Rightarrow B$ of the underlying differential bundles preserving the anchor: $T(f_0)\anc_A = \anc_B f_1$.
\end{definition}

For the definition of involution algebroid we replace the Lie bracket of a Lie algebroid (which acts on sections) with an involution which acts on a certain prolongation of the algebroid.
We now describe various differential bundles that may be thought of as prolongations of an anchored bundle.
We learnt of the following construction in section 3 of \cite{MR1861135}.

\begin{definition}[Total space of prolongation]
	If $A$ is an anchored bundle the \emph{total space of a prolongation of $A$} is the pullback
	\begin{equation*}
	\begin{tikzcd}
		\prolong \rar{\pi_1} \dar{\pi_0} & T(A) \dar{T(\pi)}\\
		A \rar{\anc}  & A
	\end{tikzcd}
	\end{equation*}
\end{definition}

\begin{remark}\label{rem:picture-of-prolongation}
	A heuristic picture of an element $(v, w)\in\prolong$ of the prolongation is
	\begin{equation*}
	\left.\begin{tikzcd}
		{} \arrow[dashed]{rr} & {} & {}\\
		{} \arrow[dashed]{rr}& {} & {}\\
		m \arrow{uu}{v} \arrow[dashed]{rr} & {} & {}
	\end{tikzcd}
	\right\} w
	\end{equation*}
	where all the dashed arrows together represent a single element $w\in T(A)$, the solid arrow represents an element $v\in A$ and $m = \pi v = \pi p w$.
	Therefore the object $\prolong$ plays an analogous role to the role played by the object of composable pairs in a groupoid.
\end{remark}

Next we describe two differential bundle structures on $\prolong$.

\begin{remark}[Bundle with $p\pi_1$]
	If $\pi:A \rightarrow M$ is an anchored bundle then the pullback
	\begin{equation*}
	\begin{tikzcd}[column sep = 0.3cm]
		(\prolong, p\pi_1) \rar \dar & (T(A), p) \dar[swap]{(T(\pi), \pi)}\\
		(A, \pi) \rar{(\anc, id)} & (T(M), p)
	\end{tikzcd}
	\end{equation*}
	in the category of differential bundles (see \ref{lem:pullbacks-in-caty-of-diff-bundles}) has lift $\lambda\times l$ and zero section $(\xi\pi, 0):A \rightarrow \prolong$.
\end{remark}

\begin{remark}[Bundle with $\pi_0$]
	If $\pi:A \rightarrow M$ is an anchored bundle then the pullback
	\begin{equation*}
	\begin{tikzcd}[column sep = 0.3cm]
		(\prolong, \pi_0) \rar \dar & (T(A), T(\pi)) \dar{(T(\pi), id)}\\
		(A, id) \rar{(\anc, \anc)} & (T(M), id)
	\end{tikzcd}
	\end{equation*}
	in the category of differential bundles (see \ref{lem:pullbacks-in-caty-of-diff-bundles}) has lift $0\times cT(\lambda)$ and zero section $(id, T(\xi)\anc):A \rightarrow \prolong$.
\end{remark}

% subsection anchored_bundles_and_prolongations (end)

\subsection{Definition of involution algebroids} % (fold)
\label{sub:definition_of_involution_algebroid}

The Lie bracket on the elements of a Lie algebra specifies the multiplication of the Lie group integrating the Lie algebra.
Roughly speaking it does this by specifying the commutator of two elements of the Lie group that are infinitesimally close to an identity element.
A similar intuition can be applied to the Lie bracket on the sections of a Lie algebroid in the case that there exists a Lie groupoid integrating the Lie algebroid.

An involution algebroid will encode the multiplication of a groupoid by specifying which composable pairs compose to the same element.
Now we present the axioms defining an involution algebroid and in the next section we demonstrate how such a structure arises on the elements of a groupoid that are infinitesimally close to an identity element.

\begin{definition}[Involution algebroid]\label{def:inv-algd}
	An \emph{involution algebroid} is an anchored bundle $\pi:A \rightarrow M$ equipped with an arrow $\alpha:\prolong \rightarrow T(A)$ such that $(\alpha, id): (\prolong, \pi_0) \Rightarrow (T(A), p)$ and $(\alpha, \anc): (\prolong, p\pi_1) \Rightarrow (T(A), T(\pi))$ are linear bundle morphisms and:-
	\begin{align*}
		T(\anc)\alpha &= cT(\anc)\pi_1 \tag{inv. algd. target}\\
		\alpha(\xi\pi, \lambda) &= \lambda\tag{inv. algd. unit} \\
		\alpha(p\pi_1, \alpha) &= \pi_1 \tag{inv. algd. inv.}\\
		% (T\alpha)(\alpha(\pi_0, \pi_1), c\pi_2) &= c(T\alpha)(\alpha(\pi_0, pc\pi_2), c (T\alpha)(\pi_1, \pi_2)) \tag{inv. algd. flip}
		(T\alpha)(\alpha(\pi_0, \pi_1), \pi_2) &= c(T\alpha)(\alpha(\pi_0, p\pi_2), c (T\alpha)(\pi_1, c\pi_2)) \tag{inv. algd. flip}
	\end{align*}
	where $\anc \pi_0 = T(\pi) \pi_1$ and $cT^2(\pi)\pi_2 = T(\anc)\pi_1$.
\end{definition}

A heuristic picture of the action of $(p\pi_1, \alpha)$ on an element $(v, w)\in\prolong$ of the prolongation is:
	\begin{equation*}
	\left.
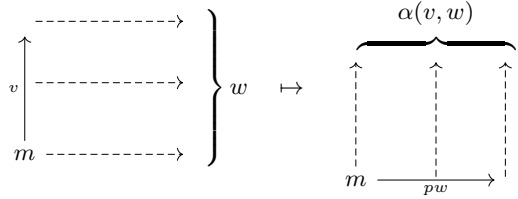
\begin{tikzcd}
		{} \arrow[dashed]{rr} & {} & {}\\
		{} \arrow[dashed]{rr}& {} & {}\\
		m \arrow{uu}{v} \arrow[dashed]{rr} & {} & {}
	\end{tikzcd}
	\right\} w \quad\mapsto \quad
	\begin{tikzcd}[column sep = -0.4cm]
	& {\alpha(v, w)} & \\[-0.7cm]
	{} & \overbrace{\hspace{2.1cm}} & {}\\
	{} & {} & {}\\
	m \arrow[dashed]{uu} \arrow{rr}[swap]{pw} & {} \arrow[dashed]{uu} & {} \arrow[dashed]{uu}
	\end{tikzcd}
	\end{equation*}
	where on each side the three dashed arrows each represent a single element of $T(A)$, the solid arrows represent elements of $A$ and $m = \pi v = \pi p w$.

\begin{remark}[Yang-Baxter]\label{rem:yang-baxter}
	We think of (inv. algd. inv.) as asserting  that the endomap $\sigma = (p\pi_1, \alpha)$ on $\prolong$ is an involution.
	When (inv. algd. flip) is rephrased in terms of $\sigma$ one obtains a Yang-Baxter style identity:
	\begin{equation*}
	(\sigma\times c)(id\times T(\sigma))(\sigma\times c) = (id \times T(\sigma))(\sigma \times c)(id\times T(\sigma))
	\end{equation*}
	which we use in section \ref{subsub:the_involution_algebroid_of_a_lie_algebroid} when proving classical Lie algebroids are examples of involution algebroids.
\end{remark}

\begin{remark}[Inv. algd. source and projection]
	Applying $T(p)$ to both sides of (inv. algd. target) we obtain
	\begin{equation*}
	T(\pi)\alpha = T(p)cT(\anc)\pi_1 = p T(\anc)\pi_1 = \anc p \pi_1 \tag{inv. algd. source}
	\end{equation*}
	which describes the source of the element $\alpha(v, w)$.
	Immediately from the fact that $(\alpha, id)$ is a bundle homomorphism we have $p\alpha = \pi_0$ which we call (inv. algd. 0).
\end{remark}

\begin{example}[Tangent bundles]
	If $M$ is an object in a tangent category then $p:T(M) \rightarrow M$ is an involution algebroid with anchor $\anc = id$ and $\alpha:T(M)\ts{id}{T(p)}T^2(M) \rightarrow T^2(M)$ given by $c\pi_1$.
\end{example}

\begin{example}[Lie algebroids]
	In \ref{subsub:the_involution_algebroid_of_a_lie_algebroid} we show directly how Lie algebroids are examples of involution algebroids.
	Now we describe another way of obtaining an involution algebroid from a Lie algebroid.
	In \ref{subsec:the_involution_algebroid_of_a_groupoid} we show how to produce an involution algebroid from a groupoid in a tangent category.
	Now the constructions in \ref{subsec:the_involution_algebroid_of_a_groupoid} only rely on infinitesimal data and so can be carried out without modification to produce an involution algebroid starting from any local Lie groupoid (defined in for instance section 2.1 of \cite{on-local-integration}).
	However the category of Lie algebroids is equivalent to the category of local Lie groupoids (theorem 1.1 of \cite{on-local-integration}).
	Therefore to every classical Lie algebroid we can associate an involution algebroid.
\end{example}

By contrast with morphisms of Lie algebroids we may define morphisms of involution algebroids simply as the structure preserving maps.
The following definition is not used much in the remainder of this paper as we are mainly concerned with constructions on a single fixed involution algebroid.
It will be important in future work when we compare the category of involution algebroids internal to the category of smooth manifolds to the category of Lie algebroids.

\begin{definition}[Morphism of involution algebroids]\label{def:morphism-inv-algd}
	If $A$ and $B$ are involution algebroids with involutions $\alpha_A$ and $\alpha_B$ respectively then \emph{a morphism $f:A \rightarrow B$ of involution algebroids} is a morphism $(f_1, f_0):A \rightarrow B$ of the underlying anchored bundles such that:
	\begin{equation*}
	\begin{tikzcd}
		A\ts{\anc_A}{T(\pi_A)}T(A) \rar{f_1\times T(f_1)} \dar{\alpha_A} & B\ts{\anc_B}{T(\pi_B)}T(B) \dar{\alpha_B}\\
		T(A) \rar{f_1} & T(B)
	\end{tikzcd}
	\end{equation*}
	commutes.
\end{definition}

\subsection{The involution algebroid of a groupoid} % (fold)
\label{subsec:the_involution_algebroid_of_a_groupoid}

In this section we start with a groupoid $G$ in a tangent category and produce an involution algebroid that is an infinitesimal approximation to $G$.
The process begins in an analogous fashion to the process described in section 3.5 of \cite{MR2157566} but in the end we construct an involution on the prolongation rather than a bracket on the sections.
So we start by constructing the underlying anchored bundle which in this case is the bundle of source constant tangent vectors to $G$ that are based at an identity element.

\begin{notation}
	In this section $G$ is a groupoid over $M$ in a tangent category $\mathbb{X}$.
	We denote by $e$, $s$, $t$ and $(-)^{-1}$ the identity, source, target and inverse maps of $G$ respectively.
\end{notation}

\begin{definition}[Anchored bundle of groupoid]\label{def:anc_bun_of_groupoid}
	The \emph{anchored bundle $\pi:A \rightarrow M$ associated to $G$} has underlying differential bundle given by the pullback 
	\begin{equation*}
	\begin{tikzcd}[column sep=1.5cm]
		(A, \pi) \rar{(\iota, e)} \dar{(\pi, id)} & (TG, p) \dar{((Ts, p),(s, id))}\\
		(M, id) \rar{((0, e), (id, e))} & (TM\times G, p\times id)
	\end{tikzcd}
	\end{equation*}
	in the category of differential bundles in $\mathbb{X}$ as described in \ref{lem:pullbacks-in-caty-of-diff-bundles}.
	The anchor is given by $T(t)\iota:A \rightarrow TM$.
\end{definition}

\begin{notation}\label{not:embed-into-TG}
	In the remainder of this section we fix $v\in A$, $w\in T(A)$ and $x\in T^2(A)$ such that $\anc v = T(\pi)w$ and $cT(\anc)w = T^2(\pi)x$.
	For notational convenience we regard:-
	\begin{itemize}
		\item $v\in T(G)$ such that $pv = em$ and $T(s)v = 0m$
		\item $w\in T^2(G)$ such that $T(p)w = T(e)T(t)v$ and $T^2(s) w = T(0)T(t)v$
	\end{itemize}
	where $m = \pi v = \pi pw = \pi pp x$.
	In \ref{def:inv-algd-of-grpd}, \ref{lem:gpd-alpha-well-typed} and \ref{lem:gpd-alpha-first-three} we use $\circ$ as infix notation for the composition $T^2(\mu):T^2(G)\ts{T^2(t)}{T^2(s)}T^2(G) \rightarrow T^2(G)$ and in \ref{lem:gpd-alpha-flip} we use $\circ$ as infix notation for the composition $T^3(\mu)$.
	In this notation the lift of the differential bundle $(A, \pi)$ of \ref{def:anc_bun_of_groupoid} acts on elements $v\in T(G)$ as $v\mapsto lv$.
	(Here we have used \ref{lem:pullbacks-in-caty-of-diff-bundles} to see $\lambda = 0\times l$.)
	Indeed $lv$ factors through $T(A)$ because $T(p)lv = 0pv = 0em = T(e)0m$.
\end{notation}

Now we show how to obtain an involution algebroid from a groupoid in a tangent category.
The flip map we use is given by:
\begin{equation}\label{eq:alpha-from-gpd}
	\alpha(v, w) = cw\circ 0v\circ(c0pw)^{-1}
\end{equation}
which we think of as the following extension:
\begin{equation*}
\begin{tikzcd}[column sep = 1cm, row sep = 1cm]
	{} \rar{cw} & {}\\
	m\uar{0v} \rar{c0pw} & {} \uar[swap, dashed]{\alpha(v, w)}
\end{tikzcd}
\end{equation*}
which corresponds to the intuition described after \ref{def:inv-algd}.
Before we verify the involution algebroid axioms we must first show that the composite in \eqref{eq:alpha-from-gpd} is well-typed.

\begin{lemma}[Composition well-typed]\label{def:inv-algd-of-grpd}
	The arrow $\alpha$ in \eqref{eq:alpha-from-gpd} is well-typed with respect to the composition in $T(G)$.
	\begin{proof}
		The composition makes sense because the equalities
		\begin{align*}
			T^2(s)0v &= 0T(s)v\\
			&= 00m\\
			&= 00tpv\\
			&= c0pT(0)T(t)v\\
			&= c0pT^2(s)w\\
			&= T^2(s)c0pw\\
			&= T^2(t)(c0pw)^{-1}
		\end{align*}
		and 
		\begin{align*}
			T^2(s) cw &= cT^2(s)w\\
			&=cT(0)T(t)v\\
			&= 0T(t)v = T^2(t) 0v
		\end{align*}
		hold.
	\end{proof}
\end{lemma}

The rest of this section is devoted to justifying definition \eqref{eq:alpha-from-gpd}.
We show that the $\alpha$ defined in \eqref{eq:alpha-from-gpd} is well-typed, a linear differential bundle morphism and satisfies the involution algebroid axioms of \ref{def:inv-algd}.

\begin{lemma}\label{lem:gpd-alpha-well-typed}
	The $\alpha$ in \eqref{eq:alpha-from-gpd} has domain $T(A)$.
	\begin{proof}
		We check that $T(p)\alpha$ is an infinitesimal variation of identity arrows:
		\begin{align*}
			T(p)\alpha(v, w) &= T(p)(cw\circ 0v\circ (c0pw)^{-1})\\
			&= T(p)cw\circ T(p)0v \circ (T(p)c0pw)^{-1}\\
			&= pw \circ 0pv \circ (pw)^{-1}\\
			&= pw \circ T(e)0m \circ (pw)^{-1}\tag{definition of $v$}\\
			&= T(e)\anc pw
		\end{align*}
		as required.
	\end{proof}
\end{lemma}

\begin{lemma}[Involution is linear I]
	The $\alpha$ defined in \eqref{eq:alpha-from-gpd} is a linear bundle morphism from $(\prolong, \pi_0)$ to $(T(A), p)$ over $id$.
	\begin{proof}
		It is easy to see that $p\alpha = \pi_0$.
		Now we check linearity.
		On the one hand
		\begin{equation*}
		 l\alpha(v, w) = lcw\circ l0v\circ (lc0pw)^{-1}
		\end{equation*}
		and on the other hand
		\begin{equation*}
		T(\alpha)(0v, cT(\lambda)w) = T(c)cT(l)w \circ T(0)0v\circ (T(c0p)cT(l)w)^{-1}
		\end{equation*}
		where we replace $\lambda$ with $l$ using \ref{not:embed-into-TG}.
		Now $l0 = T(0)0$ because $(l, 0)$ is a morphism of differential bundles.
		For the right term we check:
		\begin{align*}
			T(c)T(0)T(p)cT(l)w &= T(c)T(0)pT(l)w\\
			&= T^2(0)lpw\\
			&= lc0pw
		\end{align*}
		and for the left term:
		\begin{align*}
			cT(c)lcw &= T(l)ccw \tag{tangent category axiom}\\
			&= T(l)w\\
			&= cT(c)T(c)cT(l)w
		\end{align*}
		and so $lcw = T(c)cT(l)w$ because $cT(c)$ is an isomorphism.
		Therefore $l\alpha(v, w) = T(\alpha)(0v, cT(\lambda)w)$.
	\end{proof}
\end{lemma}

\begin{lemma}[Involution is linear II]
	The $\alpha$ defined in \eqref{eq:alpha-from-gpd} is a linear bundle morphism from $(\prolong, p\pi_1)$ to $(T(A), T(\pi))$ over $\anc$.
	\begin{proof}
		It is easy to see that $\anc p \pi_1 = T(\pi)\alpha$.
		Now we check linearity.
		On the one hand
		\begin{equation*}
		cT(\lambda)\alpha(v, w) = cT(l)cw \circ cT(l)0v \circ (cT(l)c0pw)^{-1}
		\end{equation*}
		and on the other hand
		\begin{equation*}
		T(\alpha)(\lambda v, l w) = T(c)lw \circ T(0)l v \circ (T(c0p)lw)^{-1}
		\end{equation*}
		where we replace $\lambda$ with $l$ using \ref{not:embed-into-TG}.
		The left terms are equal by a tangent category axiom and the middle terms are equal because $0$ is a natural transformation.
		For the right term:
		\begin{align*}
			cT(l)c0pw &= T(c)l0pw \tag{tangent category axiom}\\
			&= T(c)T(0)0pw\\
			&= T(c)T(0)T(p)lw
		\end{align*}
		as required.
	\end{proof}
\end{lemma}

\begin{lemma}\label{lem:gpd-alpha-first-three}
	The $\alpha$ in \eqref{eq:alpha-from-gpd} satisfies:-
	\begin{align*}
		\alpha(\xi\pi, \lambda) &= \lambda \tag{inv. algd. unit}\\
		T(\anc)\alpha &= cT(\anc)\pi_1 \tag{inv. algd. target}\\
		\alpha(p\pi_1, \alpha) &= \pi_1 \tag{inv. algd. inv.}
	\end{align*}
	\begin{proof}
		Firstly
		\begin{align*}
			\alpha(\xi\pi, \lambda)v &= c\lambda v\circ 0\xi m\circ (c0p\lambda v)^{-1}\\
			&=clv \circ 0\xi m\circ (c0pplv)^{-1}\tag{$l=\lambda$ from \ref{not:embed-into-TG}}\\
			&= lv\circ T(e)0m\circ (c00pv)^{-1}\\
			&= lv\circ T^2(e)c00m\\
			&= \lambda v\tag{$l=\lambda$ from \ref{not:embed-into-TG}}
		\end{align*}
		secondly
		\begin{align*}
			T(\anc)\alpha(v, w) &= T^2(t)(cw\circ 0v\circ (c0pw)^{-1})\\
			&= T^2(t)cw\\
			&= cT^2(t) w\\
			&= cT(\anc) w
		\end{align*}
		and thirdly
		\begin{align*}
			\alpha(pw, \alpha(v, w)) &= c\alpha(v, w)\circ 0pw\circ (c0p\alpha(v, w))^{-1}\\
			&= c\alpha(v, w)\circ 0pw\circ (c0v)^{-1}\tag{inv. algd. 0}\\
			&= c(cw \circ 0v\circ (c0pw)^{-1})\circ 0pw \circ (c0v)^{-1} \\
			&= w \circ c0v\circ (0pw)^{-1}\circ 0pw \circ (c0v)^{-1}\\
			&= w
		\end{align*}
		as required.
	\end{proof}
\end{lemma}

\begin{lemma}\label{lem:gpd-alpha-flip}
	The $\alpha$ above satisfies the flip axiom:
	\begin{equation*}
	T(\alpha)(\alpha(v, w), x) = cT(\alpha)(\alpha(v, px), cT(\alpha)(w, cx))
	\end{equation*}
	\begin{proof}
		The LHS is:
		\begin{align*}
			T&(\alpha)(\alpha(v, w), x) \\
			&= T(c)x \circ T(0)\alpha(v, w) \circ (T(c0p)x)^{-1}\\
			&=T(c)x \circ T(0)\left(cw\circ 0v \circ (c0pw)^{-1} \right)\circ (T(c0p)x)^{-1}\\
			&= T(c)x \circ T(0)cw \circ T(0)0v \circ (T(0)T(0)pw)^{-1} \circ (T(c0p)x)^{-1}
		\end{align*}
		and the RHS is:
		\begin{align*}
			c&T(\alpha)(\alpha(v, px), c T(\alpha)(w, cx))\\
			&= cT(c)cT(\alpha)(w, cx)\circ cT(0)\alpha(v, px)\circ \left(cT(c0p)cT(\alpha)(w, cx)\right)^{-1}
		\end{align*}
		We expand each of the three arrows in the composition on the RHS separately.
		The left term is:
		\begin{align*}
			c&T(c)cT(\alpha)(w, cx)\\
			&= cT(c)c(T(c)c x \circ T(0)w\circ(T(c0p)cx)^{-1})\\
			&=cT(c)cT(c)c x \circ cT(c)cT(0)w \circ (cT(c)cT(c)T(0)T(p)cx)^{-1}\\
			&= T(c)x \circ T(0)cw \circ (0cpx)^{-1}
		\end{align*}
		The middle term is:
		\begin{align*}
			c&T(0)\alpha(v, px)\\
			&= cT(0)(cpx\circ 0v\circ (c0ppx)^{-1})\\
			&=cT(0)cpx\circ cT(0)0v\circ (cT(0)c0ppx)^{-1}\\
			&=0cpx \circ T(0)0v \circ (0T(0)ppx)^{-1}
		\end{align*}
		The right term is:
		\begin{align*}
			(c&T(c0p)cT(\alpha)(w, cx))^{-1}\\
			&=(T^2(0)c\alpha(pw, pcx))^{-1}\\
			&=(T^2(0)c (cpcx\circ 0pw \circ (c0ppcx)^{-1}))^{-1}\\
			&=(T^2(0)pcx\circ T^2(0)c0pw \circ (T^2(0)0ppcx)^{-1})^{-1}\\
			&=T^2(0)0ppcx \circ (T^2(0)c0pw)^{-1} \circ (T^2(0)pcx)^{-1}\\
			&=0T(0)ppx \circ (T(0)T(0)pw)^{-1} \circ (T(c0p)x)^{-1}
		\end{align*}
		Now the result follows easily.
	\end{proof}
\end{lemma}

% subsection the_involution_algebroid_of_a_groupoid (end)

% subsection definition_of_involution_algebroid (end)

\section{Lie algebroids and involution algebroids} % (fold)
\label{sec:lie_algebroids_from_involution_algebroids}

In this section we compare involution algebroids with Lie algebroids and also describe how this comparison specialises to the case of Lie algebras.
In \ref{subsub:the_involution_algebroid_of_a_lie_algebroid} we show that every Lie algebroid is an example of an involution algebroid and in the other direction every involution algebroid in a tangent category has a Lie bracket on its set of sections (see \ref{subsubsec:the_lie_bracket_of_an_involution_algebroid} and \ref{subsubsec:Leibniz}).
In \ref{ssubub:injection_on_objects} we combine the work of \ref{subsub:the_involution_algebroid_of_a_lie_algebroid} and \ref{subsubsec:the_lie_bracket_of_an_involution_algebroid} to produce an injection on objects from the category of Lie algebroids to the category of involution algebroids.
In the same way as a Lie algebra is a Lie algebroid for which the base space is trivial we can define involution algebras as involution algebroids that have trivial base space.
In \ref{subsub:inv-alg-of-lie-alg} we write out in elementary terms how every Lie algebra is an example of an involution algebra and in \ref{subsub:equivalence_of_categories} show that there is an equivalence of categories between the category of Lie algebras and involution algebras in the category of smooth manifolds.

\subsection{Comparing brackets to involutions} % (fold)
\label{subsec:comparing_lie_brackets_to_involutions}

In this section we describe the equation that we use to construct a Lie bracket from the flip map of an involution algebroid and vice-versa.
In the sequel we need various rearrangements and special cases of this equation which we work out in this section.
To begin with suppose that we start with a Lie algebroid $\pi:A \rightarrow M$ with bracket $[-,-]$.
Then we define
\begin{equation}\label{eq:master-eqn}
	\alpha(v, w) = \alpha_L +_{T(\pi)} \alpha_R
\end{equation}
where
\[\alpha_L = T(\overline{v})\anc p w -_p(0v+_{T(\pi)}\lambda[\overline{pw}, \overline{v}])\]
and 
\[\alpha_R = T(\xi)\anc p w -_pT(\xi)\anc v+_p(w-_{T(\pi)}T(\overline{pw})\anc v)\]
where $\overline{a}$ denotes any section of $\pi$ such that $\overline{a}\pi a = a$.
In order to simplify equation \ref{eq:master-eqn} (and show that it is independent of the various choices of sections) we use the theory of connections.
The following definitions are 3.2, 4.5 and 5.2 in \cite{MR3684725}.
\begin{definition}[Vertical connection]
	A \emph{vertical connection on $\pi$} is a section $K$ of $\lambda: A \rightarrow T(A)$ such that $(K, p): T(\pi) \Rightarrow \pi$ and $(K, \pi): p \Rightarrow \pi$ are linear bundle morphisms.
\end{definition}
\begin{definition}[Horizontal connection]
	A \emph{horizontal connection on $\pi$} is a section $H$ of $(p, T(\pi)):T(A) \rightarrow A\ts{\pi}{p}T(M)$ such that $(H, id): \pi^*(p) \Rightarrow p$ and $(H, id): p^*(\pi) \Rightarrow T(\pi)$ are linear bundle morphisms.
\end{definition}
\begin{definition}[Full connection]
	A \emph{(full) connection} on $\pi$ consists of a vertical connection $K$ and horizontal connection $H$ on $\pi$ such that $KH = \xi \pi \pi_0$ and $H(p, T(\pi))+_p (\lambda K +_{T(\pi)} 0 p) = id$.
\end{definition}
If $(K, H)$ is a full connection then we have the following decomposition which is proposition 5.8 of \cite{MR3684725}.
\begin{proposition}\label{prop:decomp-full-connection}
	If $\pi$ has a full connection then the arrow
	\[(p, T(\pi), K): T(A) \rightarrow A\ts{\pi}{p} T(M)\ts{p}{\pi} A\]
	has inverse $H(\pi_0, \pi_1)+_p (\lambda\pi_2 +_{T(\pi)} 0\pi_0)$.
\end{proposition}
In terms of an arbitrary (full) connection \eqref{eq:master-eqn} simplifies to:
\begin{equation}\label{eq:master+connection}
	\alpha(v, w) = H(v, \anc pw) +_p (\lambda \alpha_2 +_{T(\pi)} 0 v)
\end{equation}
where $\alpha_2 = KT(\overline{v})\anc pw - KT(\overline{pw})\anc v + Kw - [\overline{pw}, \overline{v}]$.
Note that \eqref{eq:master-eqn} shows that \eqref{eq:master+connection} independent of the choice of connection.
If we additionally assume that $(K, H)$ is the Levi-Civita connection then $[X, Y] = KT(Y)\anc X -_{\pi} KT(X)\anc Y$ as described in section 3.1 of \cite{MR2914129}.
Therefore \eqref{eq:master-eqn} simplifies further to:
\begin{equation}\label{eq:master+levi-civita}
\alpha(v, w) = H(v, \anc pw) +_p (\lambda K w +_{T(\pi)} 0v)
\end{equation}
which also shows that \eqref{eq:master-eqn} is independent of the choice of sections extending $v$ and $pw$.
In fact under the isomorphism of \ref{prop:decomp-full-connection} $\alpha \cong (\pi_0, \anc \pi_1, \pi_2)$ and therefore 
\begin{equation}\label{eq:simple-sigma}
\sigma \cong (\pi_1, \pi_0, \pi_2)
\end{equation}
We use \eqref{eq:master+levi-civita} and \eqref{eq:simple-sigma} in \ref{subsub:the_involution_algebroid_of_a_lie_algebroid} to show that every Lie algebroid is an involution algebroid.

Next we simplify \eqref{eq:master-eqn} in the case that $A$ is a differential object as defined in \ref{def:differential-object}.
Under the isomorphism $T(A)\cong A\times A$ we have a connection where $K = \pi_1$, $H = 0\pi_0$ and $\lambda = (0, id)$ and so \eqref{eq:master+connection} becomes:
\begin{equation}\label{eq:master+diff-obj}
\alpha(v, w_H, w_V) = (v, [v, w_H] + w_V)
\end{equation}
which is the form we use in \ref{subsub:inv-alg-of-lie-alg} to show that every Lie algebra is an involution algebra.

Now we describe how to rearrange \eqref{eq:master-eqn} to obtain a Lie bracket $[-,-]$ from an flip map $\alpha$.
So let $X$ and $Y$ be sections of $\pi:A \rightarrow M$ where $A$ is an involution algebroid in a tangent category.
In \eqref{eq:master-eqn} we let $v = Xm$ and $w = T(Y)\anc v$ to obtain:
\begin{equation*}
	\alpha_R(Xm, T(Y)\anc Xm) = T(\xi)(\anc pw-_p \anc v)
\end{equation*}
and so
\begin{equation}\label{eq:master-bracket}
	(T(X)\anc Ym-_p\alpha(X, T(Y)\anc X)m)-_{T(\pi)}0Xm = \lambda [Y, X]m
\end{equation}
which is the form we use in \ref{sub:the_case_of_lie_algebroids} to show that every involution algebroid in a tangent category has a Lie bracket on its set of sections.
If in addition we use a connection $(K, H)$ then we obtain:
\begin{equation*}
[X, Y] = K\alpha(X, T(Y)\anc X)-KT(X)\anc Y
\end{equation*}
If we further restrict to the case of differential objects we see that the bracket is given in terms of $\alpha$ by:
\begin{equation}\label{eq:master-bracket-diff-obj}
	[X, Y] = \pi_1 \alpha(X, Y, 0)
\end{equation}
which could have alternatively been obtained from \eqref{eq:master+diff-obj}.

\subsection{Lie algebroids and involution algebroids} % (fold)
\label{sub:the_case_of_lie_algebroids}

Now we compare involution algebroids and Lie algebroids.
In \ref{subsub:the_involution_algebroid_of_a_lie_algebroid} we use \eqref{eq:master-eqn} to show that every Lie algebroid is an involution algebroid and in \ref{subsubsec:the_lie_bracket_of_an_involution_algebroid} we use the same equation \eqref{eq:master-eqn} to show that every involution algebroid has a Lie bracket on its set of sections.
Moreover in the presence of a unit object as described in \ref{sec:representable-units-function-algebras} we show in \ref{subsubsec:Leibniz} that the bracket on the sections of an involution algebroid satisfies the Leibniz law which is part of the classical definition of Lie algebroid.
Since \ref{subsub:the_involution_algebroid_of_a_lie_algebroid} and \ref{subsubsec:the_lie_bracket_of_an_involution_algebroid} both use the same equation we can show in \ref{ssubub:injection_on_objects} that there is an injection on objects from the category of Lie algebroids to the category of involution algebroids in smooth manifolds.

\subsubsection{The involution algebroid of a Lie algebroid} % (fold)
\label{subsub:the_involution_algebroid_of_a_lie_algebroid}

Recall that a Lie algebroid is a smooth vector bundle $\pi: A \rightarrow M$ equipped with a vector bundle morphism $\anc:A \Rightarrow T(M)$ and a Lie bracket $[-,-]:\Gamma(A)\times \Gamma(A) \rightarrow \Gamma(A)$ on the sections $\Gamma(A)$ of $\pi$ such that the following \emph{Leibniz law} holds:
\begin{equation*}
[X, f\cdot Y] = f\cdot [X, Y] + \mathcal L_{\anc X}(f)\cdot Y
\end{equation*}
where $X, Y\in \Gamma(A)$, $f\in C^{\infty}(M, \mathbb{R})$ and $(\anc X)(f)$ denotes the Lie derivative of $f$ in the direction specified by $\anc X$.
Now we show that all Lie algebroids are examples of involution algebroids with flip map given by \eqref{eq:master-eqn}.
To ease the calculations we make use of the various equivalent forms of equation \eqref{eq:master-eqn} worked out in \ref{subsec:comparing_lie_brackets_to_involutions}.

\begin{lemma}[Bundle morphism I]
	The $\alpha$ of \eqref{eq:master+levi-civita} is a linear bundle morphism $(\prolong, \pi_0) \Rightarrow (T(A), p)$ over $id$.
	\begin{proof}
		It is routine to check that $\alpha$ is a bundle morphism. 
		We now check that $\alpha$ is linear:
		\begin{align*}
			&T(\alpha)(0v, cT(\lambda)w)\\
			&= T(H)(0v, T(\anc p) cT(\lambda)w) +_{T(p)} (T(\lambda K)cT(\lambda)w +_{T^2(\pi)} T(0)0v)\\
			&= T(H)(0v, T(\anc)\lambda pw) +_{T(p)} (T(\lambda)\lambda Kw +_{T^2(\pi)} T(0)0v)\tag{3.3 (d) in \cite{MR3684725}}\\
			&= T(H)(0v, T(\anc)\lambda pw) +_{T(p)} (l\lambda Kw +_{T^2(\pi)} l0v)\tag{diff. bundle axioms}\\
			&= T(H)(0v, l \anc pw) +_{T(p)} (l\lambda Kw +_{T^2(\pi)} l0v)\tag{$\anc$ linear}\\
			&= lH(v, \anc pw) +_{T(p)} (l\lambda Kw +_{T^2(\pi)} l0v)\tag{$H$ linear}\\
			&= l\alpha(v, w) \tag{$l$ linear}
		\end{align*}
		as required.
	\end{proof}
\end{lemma}

\begin{lemma}[Bundle morphism II]
	The $\alpha$ of \eqref{eq:master+levi-civita} is a linear bundle morphism $(\prolong, p\pi_1) \Rightarrow (T(A), T(\pi))$ over $\anc$.
	\begin{proof}
		It is routine to check that $\alpha$ is a bundle morphism.
		We now check that $\alpha$ is linear:
		\begin{align*}
			T&(\alpha)(\lambda v, lw)\\
			&= T(H)(\lambda v, T(\anc p)0w)+_{T(p)}(T(\lambda K)lw +_{T^2(\pi)} T(0)\lambda v)\\
			&= cT(\lambda)H(v, \anc p w)+_{T(p)}(T(\lambda)\lambda Kw +_{T^2(\pi)} cT(\lambda)0v)\tag{$H$ is linear}\\
			&= cT(\lambda)H(v, \anc p w)+_{T(p)}(cT(\lambda)\lambda K w+_{T^2(\pi)} cT(\lambda)0v)\\
			&= cT(\lambda)\alpha(v, w)
		\end{align*}
		as required.
	\end{proof}
\end{lemma}

\begin{lemma}[Inv. algd. unit]
	The $\alpha$ of \eqref{eq:master+levi-civita} satisfies 
	$$\alpha(\xi\pi, \lambda) = \lambda$$
	\begin{proof}
		We check
		\begin{align*}
			\alpha(\xi\pi, \lambda)v &= H(\xi m, \anc p \lambda v)+_p(\lambda K\lambda v+_{T(\pi)} 0\xi m)\\
			&= H(\xi m, \anc \xi m)+_p \lambda v\\
			&= H(\xi m, 0 m)+_p \lambda v\\
			&= \lambda v \tag{$H$ linear}
		\end{align*}
		as required.
	\end{proof}
\end{lemma}

\begin{lemma}[Inv. algd. inv.]
	The $\alpha$ of \eqref{eq:master+levi-civita} satisfies 
	$$\alpha(p\pi_1, \alpha) = \pi_1$$
	\begin{proof}
		Since
		\begin{equation*}
		K\alpha(v, w) = K(H(v, \anc pw)+_p (\lambda Kw +_{T(\pi)} 0v)) = Kw
		\end{equation*}
		we can use definition of $\alpha$ twice:
		\begin{align*}
			\alpha(pw, \alpha(v, w)) &= H(pw, \anc p \alpha(v, w))+_p (\lambda K \alpha(v, w) +_{T(\pi)} 0pw)\\
			&= H(pw, \anc v) +_p (\lambda K w +_{T(\pi)} 0pw)\\
			&= w \tag{full connection}
		\end{align*}
		as required.
	\end{proof}
\end{lemma}

\begin{lemma}[Inv. algd. target]
	The $\alpha$ of \eqref{eq:master-eqn} satisfies 
	$$T(\anc)\alpha = cT(\anc)\pi_1$$
	\begin{proof}
		We treat the $\alpha_L$ and $\alpha_R$ terms separately.
		On the one hand:
		\begin{align*}
			T(\anc)\alpha_L &= T(\anc)(T(\overline{v})\anc p w -_p(0v+_{T(\pi)}\lambda[\overline{pw}, \overline{v}]))\\
			&= T(\anc\overline{v})\anc p w -_p (T(\anc)0v+_{T(p)} T(\anc)\lambda [\overline{pw}, \overline{v}])\\
			&= T(\anc\overline{v})\anc p w -_p (0\anc v+_{T(p)} l [ \anc\overline{pw}, \anc\overline{v}])\tag{Leibniz law}\\
			&= T(\anc\overline{v})\anc p w -_p (cT(\anc\overline{pw})\anc v -_p T(\anc\overline{v})\anc p w)\tag{bracket of vector fields}\\
			&= cT(\anc\overline{pw})\anc v
		\end{align*}
		and on the other hand:
		\begin{align*}
			T(\anc)\alpha_R &= T(\anc)\left(T(\xi)\anc p w -_pT(\xi)\anc v+_p(w-_{T(\pi)}T(\overline{pw})\anc v)\right)\\
			&= T(\anc\xi)\anc p w -_p T(\anc\xi)\anc v+_p(T(\anc)w-_{T(p)}T(\anc\overline{pw})\anc v)\\
			&= T(0)\anc p w -_p T(0)\anc v+_p(T(\anc)w-_{T(p)}T(\anc\overline{pw})\anc v)
		\end{align*}
		Now we use a connection $\hat{p}$ on $T(M)$ to compare $T(\anc)\alpha(v, w)$ and $cT(\anc)w$.
		First:
		\begin{align*}
			pT(\anc)\alpha(v, w) &= pT(\anc)\alpha_L(v, w)+_T(p) p T(\anc)\alpha_R(v, w) \\
			&= \anc v +_p 0m = T(\pi)w = pcT(\anc) w
		\end{align*}
		Second:
		\begin{align*}
			T(p)T(\anc)\alpha(v, w) = T(\pi)\alpha_L(v, w) = \anc pw = T(p)cT(\anc)w
		\end{align*}
		Third:
		\begin{align*}
			\hat{p}T(\anc)\alpha(v, w) &= \hat{p}T(\anc)\alpha_L+_p \hat{p}T(\anc)\alpha_R\\
			&= \hat{p}cT(\anc\overline{pw})\anc v +_p \hat{p}T(\anc)w-_p \hat{p}T(\anc\overline{pw})\anc v\\
			&= \hat{p}T(\anc)w = \hat{p}cT(\anc)w
		\end{align*}
		and so $cT(\anc)w = T(\anc)\alpha(v, w)$ because their images under the isomorphism $(p, T(p), \hat{p})$ are the same.
	\end{proof}
\end{lemma}

\begin{lemma}[Inv. algd. flip]\label{lem:flip-for-lie-algd}
	The $\sigma$ of \eqref{eq:simple-sigma} satisfies
	\begin{equation*}
		(\sigma\times c)(id\times T(\sigma))(\sigma\times c) = (id\times T(\sigma))(\sigma\times c)(id\times T(\sigma))
	\end{equation*}
	\begin{proof}
		Using \ref{prop:decomp-full-connection} the object $\prolong\ts{T(\anc)}{T^2(\pi)}T^2(A)$ is isomorphic to $\mathcal L^2(A)\cong A\ts{\pi}{\pi}A\ts{\pi}{\pi}A\ts{\pi}{\pi}A\ts{\pi pi}A\ts{\pi}{\pi}A\ts{\pi}{\pi}A$.
		Next we convert the arrows $\sigma \times c$ and $id\times T(\sigma)$ into elements of the symmetric group $S_7$ corresponding to how they permute the factors of $\mathcal L^2(A)$.
		Since
		\begin{equation*}
		(\sigma\times c)(a_0, a_1, a_2, a_3, a_4, a_5, a_6) = (a_1, a_0, a_2, a_3, a_5, a_4, a_6)
		\end{equation*}
		the arrow $\sigma\times c$ corresponds to the permutation $(01)(45)$.
		Similarly since
		\begin{equation*}
		(id\times T(\sigma))(a_0, a_1, a_2, a_3, a_4, a_5, a_6) = (a_0, a_3, a_4, a_1, a_2, a_5, a_6)
		\end{equation*}
		the arrow $id\times T(\sigma)$ corresponds to the permutation $(13)(24)$.
		Therefore we check
		\begin{align*}
			(\sigma\times c)(id\times T(\sigma))(\sigma\times c) &\leftrightarrow (01)(45)(13)(24)(01)(45)\\
			&= (01)(13)(01)(45)(24)(45)\\
			&= (03)(25)\\
			&= (13)(01)(13)(24)(45)(24)\\
			&= (13)(24)(01)(45)(13)(24)\\
			&\leftrightarrow (id\times T(\sigma))(\sigma\times c)(id\times T(\sigma))
		\end{align*}
		and conclude that the flip axiom holds.
	\end{proof}
\end{lemma}

\begin{corollary}
	Every Lie algebroid is an involution algebroid with flip map $\alpha$ given by \eqref{eq:master-eqn}.
\end{corollary}

% subsection the_involution_algebroid_of_a_lie_algebroid (end)

\subsubsection{Lie bracket of an involution algebroid}
\label{subsubsec:the_lie_bracket_of_an_involution_algebroid}

In this section we work in a tangent category without any specified unit
object $R$.
We fix an involution algebroid $(\pi: A \to M, \anc, \alpha)$ and describe a Lie bracket on the sections $\Gamma(\pi)$ of $\pi$.
This Lie bracket is induced from the Lie bracket of vector fields on the total space $A$ of the involution algebroid.
We begin by using the involution to define a monic additive map $\Gamma(\pi) \to \chi(A)$ where $\chi(A)$ is the set of vector fields on $A$.

\begin{lemma}
	The function $\alpha_{(-)} : \Gamma(\pi) \to \chi(A)$ given by:
  \[
    \alpha_{X} = \alpha ( id, T(X)\anc )
  \]
  is well-defined.
  \begin{proof}
  Indeed $p \alpha ( id, T(X)\anc ) = id$ by (inv. algd. 0).
  \end{proof}
\end{lemma}

\begin{lemma}\label{lem:transfer-monic}
  The morphism $\alpha_{(-)}$ is additive and monic.
\end{lemma}
\begin{proof}
  The map preserves addition:
  \begin{align*}
    \alpha_{X+_{\pi}Y}v
    &= \alpha (v , T(X +_\pi Y)\anc v) \\
    &= \alpha (v, (T(X)\anc v) +_{T\pi} (T(Y)\anc v)) \\
    &= \alpha (v, T(X)\anc v) +_p \alpha (v, T(Y)\anc v)
    & \tag{$\alpha$ linear} \\
    &= \alpha_X +_p \alpha_Y
  \end{align*}
  and zero:
  \begin{align*}
    \alpha_\xi v
    &= \alpha(v, T(\xi)\anc v) \\
    &= \alpha (id, T(\xi)) (v, \anc v) \\
    &= 0v & \tag{$\alpha$ linear}
  \end{align*}
  To show the map is monic let $X,Y\in \Gamma(\pi)$ satisfying $\alpha_X = \alpha_Y$. Then:
  \begin{align*}
  	\alpha(p\pi_1, \alpha_X ) &= \alpha(p\pi_1, \alpha_Y )\\
  	\implies \alpha ( p \pi_1 , \alpha ) ( id, T(X) \anc) &= \alpha ( p \pi_1 , \alpha ) ( id, T(Y) \anc)\\
  	\implies T(X)\anc v &= T(Y)\anc v \tag{inv. algd. inv.}\\
  	\implies pT(X)\anc &= pT(Y)\anc\\
  	\implies X\pi = Y\pi
  \end{align*}
  and so $X=Y$ because $\pi$ is an epimorphism.
\end{proof}
We record some results on the form of $\alpha_X$ and how it interacts with the tangent functor and
canonical flip on $T^2A$.
\begin{lemma}\label{lemma:alpha-comp}
  If $X,Y \in \Gamma(\pi)$ then
  $$cT(\alpha_Y)\alpha_Xv = T(\alpha)(\alpha_Yv, cT(\alpha_Y)T(X)\anc v)$$
\begin{proof}
	We compute
  \begin{align*}
    & cT(\alpha_Y)\alpha_X v\\
    =& c T(\alpha)(\alpha_X v, T^2(Y)T(\anc)\alpha_X v) \\
    =& c T(\alpha)(\alpha(v, T(X)\anc v), T^2(Y)T(\anc)\alpha_X v) \\
    =& T(\alpha)(\alpha(v, pT^2(Y)T(\anc)\alpha_X v),
      c T(\alpha) (T(X)\anc v, cT^2(Y)T(\anc)\alpha_X v))
     & \tag{inv. algd. flip} \\ % branch
    =& T(\alpha)(\alpha(v, T(Y)\anc p \alpha_X v),   %begin branch 1
       c T(\alpha) (T(X)\anc v, cT^2(Y)T(\anc)\alpha_X v)) 
       & \tag{naturality $p$}
    \\=& T(\alpha)(\alpha(v, T(Y)\anc v),
         c T(\alpha) (T(X)\anc v, cT^2(Y)T(\anc)\alpha_X v))
         & \tag{bundle morphism}
    \\=& T(\alpha)(\alpha_Yv,
         c T(\alpha) (T(X)\anc v, cT^2(Y)T(\anc)\alpha_X v))
     & \tag{definition} % end branch 1
    \\=& T(\alpha)(\alpha_Yv,                                 %begin branch 2
         cT(\alpha)(T(X)\anc v, T^2(Y)cT(\anc)\alpha(v, T(X)\anc v)))
     & \tag{naturality $c$}
    \\=& T(\alpha)(\alpha_Yv,
         cT(\alpha)(T(X)\anc v, T^2(Y)ccT(\anc)T(X)\anc v))
     &\tag{inv. algd. target}
    \\=& T(\alpha)(\alpha_Yv,
       cT(\alpha)( 1, T^2(Y)T(\anc) ) T(X) \anc v)
    \\=& T(\alpha)(\alpha_Yv,  cT(\alpha_Y)T(X)\anc v)
  \end{align*}
  as required.
\end{proof}
\end{lemma}

Now we show that \eqref{eq:master-bracket} defines a Lie bracket on the sections of an involution algebroid $(\pi, \anc, \alpha)$.
Recall that \eqref{eq:master-bracket} is:
\begin{equation}\label{eq:lie-bracket-of-algd}
    [X,Y]_\alpha :=
    \alpha_YX\sd T(X)\anc Y 
\end{equation}
where we have used the notation of \ref{sub:affine_structures}.
\begin{lemma}
  The Lie bracket of \eqref{eq:lie-bracket-of-algd} is well defined.
\begin{proof}
  We check that
  \begin{align*}
    T(\pi)\alpha_YX &= T(\pi)\alpha(X, T(Y)\anc X)\\
    &= \anc pT(Y)\anc X \tag{inv. algd. source}\\
    &= \anc Y \tag{$p\anc = \pi$}
  \end{align*}
  as required.
\end{proof}
\end{lemma}

We now use \ref{lemma:alpha-comp} to relate the Lie bracket $[-,-]_A$ on the total space $A$ to the bracket $[-,-]_{\alpha}$ we've defined using $\alpha$.

\begin{proposition}\label{prop:alpha-sub-preserves-bracket}
  The map $\alpha_{(-)}$ preserves the bracket:
  \[
    \alpha_{[X,Y]_\alpha} = [\alpha_X, \alpha_Y]_A
  \]
\begin{proof}
  We postcompose with $l$ and expand:
  \begin{align*}
    & l [\alpha_X, \alpha_Y]_A \\
    =& l (cT(\alpha_Y)\alpha_X \sd T(\alpha_X)\alpha_Y ) \\
    =& (T(\alpha)(\alpha_Y, cT(\alpha_Y)T(X)\anc -_p T(\alpha)(\alpha_Y, T^2(X)T(\anc)\alpha_Y))-_{Tp} 0\alpha_X\tag{lemma \ref{lemma:alpha-comp}}\\
    =&T(\alpha)(\alpha_Y - \alpha_Y,cT(\alpha_Y)T(X)\anc  -_p T^2(X)T(\anc)\alpha_Y)-_{Tp} 0\alpha_X \tag{$\alpha$ linear}\\
    =& T(\alpha)(\xi\pi,cT(\alpha_Y)T(X)\anc-_p T^2(X)cT(\anc)T(Y)\anc) )-_{Tp} 0\alpha_X\tag{inv. alg. targ.}\\
    =& T(\alpha)(\xi\pi, c(T(\alpha_Y)T(X)\anc -_{Tp} T^2(X)T(\anc)T(Y)\anc    ) ))-_{Tp} T(\alpha)(0, 0T(X)\anc ) \tag{$c$ natural, additive}\\
    =& T(\alpha)(\xi\pi,
       c(T(\alpha_Y)T(X)\anc -_{Tp} T^2(X)T(\anc)T(Y)\anc) -_{T^2\pi} 0T(X)\anc )
    \tag{$\alpha$ linear}\\
    =& T(\alpha)(\xi\pi,
       c((T(\alpha_Y)T(X) -_{Tp} T^2(X)T(\anc)T(Y)) -_{T^2\pi} T(0)T(X)) \anc)
    \tag{$c$ additive}\\
    =& T(\alpha)(\xi\pi, cT(\lambda)T[X,Y]_\alpha\anc )& \tag{by \eqref{eq:master-bracket}} \\
    =& l \alpha(id, T[X,Y]_\alpha\anc) = l \alpha_{[X,Y]_\alpha} & \tag{linearity}
  \end{align*}
  and so $\alpha_{[X,Y]_\alpha} = [\alpha_X, \alpha_Y]_A$ because $l$ is a monomorphism.
\end{proof}
\end{proposition}

\begin{theorem}
	The bracket defined in \eqref{eq:lie-bracket-of-algd} is a Lie bracket.
\end{theorem}
\begin{proof}
	By using \ref{lem:transfer-monic} we can leverage the standard results concerning the bracket of vector fields found in 3.4 of \cite{MR3192082}.
	First we show that the bracket is bilinear:
  \begin{align*}
    \alpha_{[X + Y, Z]}
    &= [\alpha_{X + Y}, \alpha_Z]_A \\
    &= [\alpha_X + \alpha_Y, \alpha_Z]_A \\
    &= [\alpha_X, \alpha_Z]_A + [\alpha_Y, \alpha_Z]_A \\
    &= \alpha_{[X,Z]} + \alpha_{[Y,Z]} \\
    &= \alpha_{[X,Z] + [Y,Z]}
  \end{align*}
  second that it is alternating:
  \begin{align*}
    \alpha_{[X,Y]} &= [\alpha_X, \alpha_Y] \\
    &= -[\alpha_Y, \alpha_X] \\
    &= -\alpha_{[Y,X]} = \alpha_{-[Y,X]}
  \end{align*}
  and third that the Jacobi identity holds:
  \begin{align*}
    &\alpha_{ [X,[Y,Z] + [Z,[X,Y]] + [Y, [Z,X]] }\\
    =& \alpha_{[X,[Y,Z]]} + \alpha_{[Z,[X,Y]} + \alpha_{[Y, [Z,X]} \\
    =& [\alpha_X, \alpha_{[Y,Z]}]_A + [\alpha_Z, \alpha_{[X,Y]}]_A + [\alpha_Y, \alpha_{[Z,X]}]_A \\
    =& [\alpha_X, [\alpha_Y, \alpha_Z]_A]_A + [\alpha_Z, [\alpha_X, \alpha_Y]_A]_A
       + [\alpha_Y, [\alpha_Z, \alpha_X]_A]_A \\
    =& 0_A = \alpha_\xi
  \end{align*}
  Therefore $[-,-]$ is a Lie bracket.
\end{proof}

Recall that for a Lie algebroid the identity $\anc[X, Y] = [\anc X, \anc Y]$ follows from the Leibniz law.
We end this section by demonstrating that this equality holds even when we don't assume the presence of a unit object (which is required to formulate the Leibniz law).

\begin{proposition}
  Let $X,Y \in \Gamma(\pi)$, then $\anc[X,Y]_\alpha = [\anc X, \anc Y]_A$.
\end{proposition}
\begin{proof}
  We check
  \begin{align*}
    l \anc [X,Y]_\alpha =& T(\anc)\lambda[X,Y]_{\alpha} \\
    =& T(\anc)((T(X)\anc Y -_{p} \alpha_YX) -_{T\pi} 0X) \\
    =& (T(\anc)T(X)\anc Y -_{p} T(\anc)\alpha_YX) -_{Tp} T(\anc)0X \\
    =& (T(\anc)T(X)\anc Y -_{p} cT(\anc)T(Y)\anc X) -_{Tp} 0\anc X
     & \tag{inv. alg. target} \\
    =& l [\anc X, \anc Y]
  \end{align*}
  as required.
\end{proof}

\subsubsection{The Leibniz law}
\label{subsubsec:Leibniz}

In this section we prove that the Lie bracket on an involution
algebroid constructed in \ref{subsubsec:the_lie_bracket_of_an_involution_algebroid} satisfies the Leibniz law when we additionally assume presence of a unit object.
Recall that the Leibniz law is:
\begin{equation*}
[X, f\bullet_{\pi} Y] = f\bullet_{\pi} [X, Y] + \mathcal L_{\anc X}(f)\bullet_{\pi} Y
\end{equation*}
where $\mathcal L_{\anc X}(f) = \hat{p}T(f)\anc X$ and $\hat{p}$ is from \ref{def:differential-object}.
We break the proof of the Leibniz law into two stages.
First we prove a lemma that describes how $\alpha_{(-)}$ interacts with the multiplication $\bullet$.

\begin{lemma}\label{lem:leibniz-lemma}
	\begin{equation*}
	\alpha(X, T(\psi\bullet_{\pi} Y)\anc X) -_{T(\pi)} \psi\bullet_p\alpha(X,T(Y)\anc X) = \lambda\mathcal L _{\anc X}(\psi)\bullet_{\pi} Y
	\end{equation*}
	\begin{proof}
		We aim to use the equation
		\begin{equation*}
			r\bullet^T_{\pi}w -_p pr\bullet_{T(\pi)} w = \hat{p}r\bullet_{T(\pi)}\lambda p w -_p T(\xi) 0m
		\end{equation*}
		which is \eqref{eq:derivative-of-scalar-multiplication} in \ref{sec:representable-units-function-algebras} and compute:
		\begin{align*}
			\alpha&(X, T(\psi\bullet_{\pi} Y)\anc X) -_{T(\pi)} \psi\bullet_p\alpha(X,T(Y)\anc X)\\
			&= \alpha(X, T(\psi\bullet_{\pi} Y)\anc X) -_{T(\pi)} \alpha(X,\psi\bullet_{T(\pi)}T(Y)\anc X)\tag{$\alpha$ linear}\\
			&= \alpha(\xi, (T(\psi\bullet_{\pi} Y)-_p \psi p\bullet_{T(\pi)}T(Y))\anc X)\tag{$\alpha$ linear}\\
			&= \alpha(\xi, (\hat{p}T(\psi)\bullet_{T(\pi)}\lambda pT(Y)-_p T(\xi))\anc X)\tag{by \eqref{eq:derivative-of-scalar-multiplication}}\\
			&= \alpha(\xi, \hat{p}T(\psi)\bullet_{T(\pi)}\lambda p T(Y)\anc X)-_{T(\pi)}\alpha(\xi, T(\xi)\anc X) \tag{$\alpha$ linear}\\
			&= \alpha(\xi, \hat{p}T(\psi)\bullet_{T(\pi)}\lambda p T(Y)\anc X)\tag{$\alpha$ preserves zero}\\
			&= \hat{p}T(\psi)\bullet_{T(\pi)}\lambda pT(Y)\anc X\tag{inv. algd. unit}\\
			&= \lambda\mathcal L _{\anc X}(\psi)\bullet_{\pi} Y
		\end{align*}
		as required.
	\end{proof}
\end{lemma}

Using \ref{lem:leibniz-lemma} we can now give a direct proof of the Leibniz law.
\begin{proposition}
  The bracket of \ref{eq:master-bracket} satisfies the Leibniz law.
\begin{proof}
	We aim to use \ref{lem:leibniz-lemma} and compute:
	\begin{align*}
		\lambda&([X, \psi\bullet_{\pi} Y]-\psi\bullet_{\pi}[X, Y])\\
		&=\lambda[X, \psi\bullet_{\pi} Y] -_{T(\pi)} \psi\bullet_{\pi}[X, Y]\\
		&= \left(\psi\bullet_p T(X)\anc Y -_p \psi\bullet_p \alpha(X, T(Y)\anc X\right) \\
		&-_{T(\pi)} \left(\psi\bullet_{p}T(X)\anc Y -_p \alpha(X, T(\psi\bullet_{\pi}Y)\anc X)\right)\tag{by \eqref{eq:master-bracket}}\\
		&= \alpha(X, T(\psi\bullet_{\pi} Y)\anc X)-_{T(\pi)} \psi\bullet_p\alpha(X, T(Y)\anc X) \tag{interchange law}\\
		&= \lambda\mathcal L _{\anc X}(\psi)\bullet_{\pi} Y
	\end{align*}
	and so $[X, \psi\bullet_{\pi} Y] = \psi\bullet_{\pi}[X, Y] + \lambda\mathcal L _{\anc X}(\psi)\bullet_{\pi} Y$ because $\lambda$ is monic.
\end{proof}
\end{proposition}

\subsubsection{Injection on objects} % (fold)
\label{ssubub:injection_on_objects}

In this section we confirm that the work in sections \ref{subsub:the_involution_algebroid_of_a_lie_algebroid} and \ref{subsubsec:the_lie_bracket_of_an_involution_algebroid} imply an injection on objects from the category $LieAlgd$ of Lie algebroids to the category $SInvAlgd$ of involution algebroids in the category of smooth manifolds.
So let $\pi:A \rightarrow M$ be a Lie algebroid and $X$ and $Y$ sections of $\pi$.
Applying the function defined in \ref{subsub:the_involution_algebroid_of_a_lie_algebroid} to $A$ we obtain an involution algebroid $\iota(A)$.
Applying the function defined in \ref{subsubsec:the_lie_bracket_of_an_involution_algebroid} we obtain another Lie algebroid $r(\iota(A))$.
By construction the underlying anchored bundle of $r(\iota(A))$ is the same as the underlying anchored bundle of $A$.

Now we compare the bracket $[-,-]_2$ on $r(\iota(A))$ with the original bracket $[-,-]$ on $A$.
First we see that 
\begin{equation}\label{eq:bracket-on-inv-appendix}
	\lambda [X, Y]_2 = \left(T(X)\anc Y -_p \alpha(X, T(Y)\anc X\right))+_{T(\pi)}0X
\end{equation}
where $\alpha$ is in the first instance defined in terms of $[-,-]$ by \eqref{eq:master-eqn}.
However since we are applying $\alpha$ to vector fields (rather than elements of the prolongation) we see that we are in the situation described in \eqref{eq:master-bracket} where $\alpha_R$ vanishes.
But now it is easy to see that when we substitute \eqref{eq:master-bracket} into \ref{eq:bracket-on-inv-appendix} we obtain $[X, Y]_2 = [X, Y]$.

% subsection injection_on_objects (end)

\subsection{Lie algebras and involution algebras} % (fold)
\label{sub:the_case_of_lie_algebras}

In this section we describe how our previous results specialise to the case when our base space $M$ is the terminal object in the category $Man$ of smooth manifolds.
In classical Lie theory this corresponds to restricting attention from the category of Lie algebroids to the category $LieAlg$ of Lie algebras.
Correspondingly we define a smooth involution algebra as an involution algebroid in $Man$ with trivial base space and denote by $SInvAlg$ the category of smooth involution algebras.
The main result is \ref{cor:equiv-with-lie-algebras} which exhibits an equivalence of categories between $SInvAlg$ and $LieAlg$.
In \ref{subsub:inv-alg-of-lie-alg} we describe how every Lie algebra has the structure of a involution algebra.
This provides a function from the objects of $LieAlg$ to the objects of $SInvAlg$.
In \ref{subsub:equivalence_of_categories} we show how this function can be extended to a functor $LieAlg \rightarrow SInvAlg$ and then prove that this functor is full, faithful and essentially surjective.

\subsubsection{The involution algebra of a Lie algebra}
\label{subsub:inv-alg-of-lie-alg}

Recall that a Lie algebra is a vector space $A$ equipped with a bilinear and anti-symmetric bracket $[-,-]:A\times A \rightarrow A$ such that
\begin{align*}
	[[x, y], z] &= [[x, z], y] + [x, [y, z]]\tag{Jacobi}
\end{align*}
where $x, y, z\in A$.
Therefore a Lie algebra is a Lie algebroid with the trivial base space $\mathbb{R}^0$.
In this section we show that Lie algebras are examples of involution algebroids with trivial base space.
Strictly speaking these results follow from their counterparts in \ref{subsub:the_involution_algebroid_of_a_lie_algebroid} but we retain them here because in the case of Lie algebras it is possible to give an explicit algebraic proof without relying on the Levi-Civita connection.
Recall that a \emph{differential object} (\ref{def:differential-object}) is a differential bundle over the terminal object.

\begin{definition}[Involution algebra]
	An \emph{involution algebra} is an involution algebroid over the terminal object. (I.e. for which $M\cong 1$.)
	%  differential object $A$ equipped with a map $\alpha:A\times T(A) \rightarrow T(A)$ such that $(\alpha, id): (\prolong, \pi_0) \Rightarrow (T(A), p)$ and $(\alpha, \anc): (\prolong, p\pi_1) \Rightarrow (T(A), T(\pi))$ are linear bundle morphisms and:-
	% \begin{align*}
	% 	\alpha(\xi\pi, \lambda) &= \lambda\tag{inv. algd. unit} \\
	% 	\alpha(p\pi_1, \alpha) &= \pi_1 \tag{inv. alg. inv.}\\
	% 	T(\alpha)(\alpha(v, w), x) &= cT(\alpha)(\alpha(v, px), cT(\alpha)(w, cx)) \tag{inv. alg. flip}
	% \end{align*}
	% where $v\in A$, $w\in T(A)$ and $x\in T^2(A)$ such that $\anc v = T(\pi)w$ and $T(\anc)w = cT^2(\pi)x$.
\end{definition}

\begin{remark}
	Under the isomorphism $T(A)\cong A\times A$ of \ref{def:differential-object} an involution algebra is equivalently a differential object $A$ equipped with a map $\alpha:A\times A\times A \rightarrow A\times A$ such that $(\alpha, id):(A\times A\times A, \pi_0) \Rightarrow (A\times A, \pi_0)$ and $(\alpha, !): (A\times A\times A, \pi_1) \Rightarrow (A\times A, !)$ are linear bundle morphisms and:-
	\begin{align*}
		\alpha(0, 0, v) = (0, v) \tag{inv. alg. unit}\\
		\alpha(w_H, \alpha(v, w_H, w_V)) = (w_H, w_V) \tag{inv. alg. inv.}\\
		T(\alpha)(\alpha(v, w_H, w_V), x_H, x_V, x'_H, x'_V)\\
		=cT(\alpha)(\alpha(v, x_H, x_V), cT(\alpha)(w_H, w_V, x_H, x'_H, x_V, x'_V))\tag{inv. alg. flip}
	\end{align*}
	where $0, v, w_H, w_V, x_H, x_V, x'_H, x'_V\in A$.
\end{remark}

Now we show that every classical Lie algebra is an involution algebra in the category of smooth manifolds.
We use \eqref{eq:master+diff-obj}:
\begin{equation*}
\alpha(v, w_H, w_V) = (v, w_V + [v, w_H])
\end{equation*}
to check the involution algebra axioms.
Note that by differentiating the definition of $\alpha$ we obtain
\begin{equation*}
T(\alpha)(a, b, c, d, e, f) = (a, d + [a, c], b, f + [a, e] + [b, c])
\end{equation*}
where $a, b, c, d, e, f\in A$.

\begin{lemma}[Bundle morphism I]
	The $\alpha$ of \eqref{eq:master+diff-obj} is a linear bundle morphism $(\prolong, \pi_0) \Rightarrow (T(A), p)$ over $id$.
	\begin{proof}
		The equation $p\alpha = \pi_0$ follows immediately from \eqref{eq:master+diff-obj}.
		We now check linearity:
		\begin{align*}
			T(\alpha)(0\times cT(\lambda))(v, w_H, w_V) &= T(\alpha)(v, 0, 0, 0, w_H, w_V) \\
			&= (v, 0, 0, w_V+[v, w_H])\\
			&= l\alpha(v, w_H, w_V)
		\end{align*}
		as required.
	\end{proof}
\end{lemma}

\begin{lemma}[Bundle morphism II]
	The $\alpha$ of \eqref{eq:master+diff-obj} is a linear bundle morphism $(\prolong, p\pi_1) \Rightarrow (T(A), !)$ over $!:A \rightarrow 1$.
	\begin{proof}
		The map $(\alpha, !)$ is clearly a bundle morphism.
		Now we check linearity:
		\begin{align*}
			T(\alpha)(\lambda\times l)(v, w_H, w_V)&= T(\alpha)(0, v, w_H, 0, 0, w_V)\\
			&= (0, 0, v, w_V+[v, w_H])\\
			&= cT(\lambda)\alpha(v, w_V+[v, w_H])
		\end{align*}
		as required.
	\end{proof}
\end{lemma}

\begin{lemma}[Inv. alg. unit]
	The $\alpha$ of \eqref{eq:master+diff-obj} satisfies:
	\begin{equation*}
	\alpha(\xi\pi, \lambda) = \lambda
	\end{equation*}
	\begin{proof}
		We check $\alpha(\xi m, \xi m, v) = (\xi m, [\xi m, \xi m]+v) = (\xi m, v)$.
	\end{proof}
\end{lemma}

\begin{lemma}[Inv. alg. inv.]
The $\alpha$ of \eqref{eq:master+diff-obj} satisfies:	$\alpha(p\pi_1, \alpha) = \pi_1$
	\begin{proof}
		Using the definition of $\alpha$ twice:
		\begin{align*}
			\alpha(w_H, \alpha(v, w_H, w_V)) &= \alpha(w_H, v, w_V + [v, w_H])\\
			&= (w_H, w_V + [v, w_H] +[w_H, v])\\
			&= (w_H, w_V)
		\end{align*}
		therefore (inv. alg. inv.) holds.
	\end{proof}
\end{lemma}

\begin{lemma}[Inv. alg. flip]
The $\alpha$ of \eqref{eq:master+diff-obj} satisfies:
	\[T(\alpha)(\alpha(v, w), x) = cT(\alpha)(\alpha(v, px), cT(\alpha)(w, cx))\]
	\begin{proof}
		On the one hand:
		\begin{align*}
			&T(\alpha)(\alpha(v, w_H, w_V), x_H, x_V, x'_H, x'_V)\\
			&= T(\alpha)(v, w_V+[v, w_H], x_H, x_V, x'_H, x'_V)\\
			&= (v, x_V+[v, x_H], w_V+[v, w_H], x'_V+[v, x'_H]+[w_V, x_H]+[[v, w_H], x_H])
		\end{align*}
		and on the other hand:
		\begin{align*}
			&cT(\alpha)(\alpha(v, x_H, x_V), cT(\alpha)(w_H, w_V, x_H, x'_H, x_V, x'_V))\\
			&= cT(\alpha)(v, x_V+[v, x_H], c(w_H, x'_H+[w_H, x_H], w_V, x'_V+[w_H, x_V]+[w_V, x_H]))\\
			&= cT(\alpha)(v, x_V+[v, x_H], w_H, w_V, x'_H+[w_H, x_H], x'_V+[w_H, x_V]+[w_V, x_H])\\
			&= c(v, w_V+[v, w_H], x_V+[v, x_H], z)\\
			&= (v, x_V+[v, x_H], w_V+[v, w_H], z)
		\end{align*}
		where
		\begin{align*}
			&z= x'_V+[w_H, x_V]+[w_V, x_H]+[v, x'_H + [w_H, x_H]]+[x_V+[v, x_H], w_H]\\
			&= x'_V+[v, x'_H]+[w_V, x_H]+[[v, w_H], x_H]
		\end{align*}
		where the last equality uses the anti-symmetry property and Jacobi identity of the Lie bracket.
		Now the required equality follows easily.
	\end{proof}
\end{lemma}

\subsubsection{Equivalence of categories} % (fold)
\label{subsub:equivalence_of_categories}

Let $SInvAlg$ denote the category of involution algebras in the category $Man$ of smooth manifolds.
In this section we show that $SInvAlg$ is equivalent to the category of Lie algebras.

\begin{lemma}[Inclusion functor]\label{def:inclusion-functor}
	There is a functor $\iota:LieAlg \rightarrow SInvAlg$ that takes a Lie algebra $A$ to the involution algebra on $A$ defined by \eqref{eq:master+diff-obj} and that is the identity on arrows.
	\begin{proof}
		In \ref{subsub:inv-alg-of-lie-alg} we showed that $\iota(A)$ is an involution algebra.
		Next we check that if $f:A\rightarrow B$ is a morphism of Lie algebras then
		\begin{align*}
			\alpha_B(f\times f\times f)(a, b, c) &= \alpha_B(fa, fb, fc)\\
			&= (fb, [fa, fb] +fc)\\
			&= (f\times f)\alpha_A(a, b, c)
		\end{align*}
		and so $f$ is also a morphism of involution algebras.
		Finally it is clear that the assignment respects composition and identities.
	\end{proof}
\end{lemma}
Before we prove that $\iota$ induces an equivalence of categories we recall that we can express the bracket in terms of $\alpha$ using \eqref{eq:master-bracket-diff-obj}. 
Explicitly: if $A$ is an involution algebra with flip $\alpha$ we denote by $r(A)$ the Lie algebra with bracket given by $[x, y] = \pi_1\alpha(x, y, 0)$.

\begin{lemma}[Injective on objects]
	The functor defined in \ref{def:inclusion-functor} is injective on objects and determines a full subcategory of $SInvAlg$.
	\begin{proof}
		It has retraction $r$.
	\end{proof}
\end{lemma}

\begin{lemma}[Essentially surjective]
	The functor defined in \ref{def:inclusion-functor} is essentially surjective.
	\begin{proof}
		Suppose that $A$ is an involution algebra in $Man$ with flip $\alpha$.
		Using the results of \ref{subsubsec:the_lie_bracket_of_an_involution_algebroid} we know that $r(A)$ is a Lie algebra.
		But $\iota(r(A))$ has flip that sends $(a, b, c)$ to:
		\begin{align*}
			(b, \pi_1\alpha(a, b, 0)+c) &= (b, \pi_1\alpha(a, b, 0)+\pi_1\alpha(a, 0, c) )\tag{inv. alg. unit}\\
			&= (b, \pi_1\alpha(a, b, 0)+\pi_1\alpha(a, 0, 0)+\pi_1\alpha(0, 0, c))\tag{$\alpha$ linear}\\
			&= (b, \pi_1\alpha(a, b, 0)+\pi_1\alpha(a, 0, c))\\
			&= \alpha(a, b, c)
		\end{align*}
		therefore $A = \iota(r(A))$ and so $\iota$ is essentially surjective.
	\end{proof}
\end{lemma}

\begin{corollary}\label{cor:equiv-with-lie-algebras}
	The category of Lie algebras is equivalent to the category of involution algebras in $Man$.
\end{corollary}

% subsubsection equivalence_with_lie_algebras (end)

\section{Homotopy theory of involution algebroids} % (fold)
\label{sec:the_homotopy_theory_of_involution_algebroids}

In section 5.1 of \cite{MR1973056} the \emph{Weinstein local groupoid} of a Lie algebroid is constructed in an analogous way to how the fundamental groupoid is constructed from a manifold.
The appropriate definition of paths and homotopies in a Lie algebroid required for this construction are the \emph{admissible paths} and \emph{admissible homotopies} of section 1 in \cite{MR1973056}.
In fact this construction gives an equivalence of categories
\begin{equation*}
\begin{tikzcd}
	LieAlgd \rar[yshift=0.3cm]{w} \rar[phantom]{\perp} & LocGpd \lar[yshift=-0.3cm]{alg}
\end{tikzcd}
\end{equation*}
where $alg$ is the extension of section 3.5 of \cite{MR2157566} to local Lie groupoids and $w$ is the Weinstein local groupoid construction.
Therefore one approach to understanding the morphisms in $LieAlgd$ is to understand the homotopy theory of Lie algebroids.

In this section we describe the homotopy theory associated to involution algebroids.
Moreover we analyse some special cases of paths and homotopies in involution algebroids that arise in the composite $alg \circ w$ in the diagram above.
In a future paper we complete this picture by examining the extra assumptions required to integrate involution algebroids to groupoids.
In this section we will assume the existence of various function spaces as we need them.
Again this assumption is not an unreasonable one due to the embedding theorem of \cite{MR3725887}.
In this section we also require the use of a complete curve object $R$ as described in \ref{sub:curve_objects} that is also a unit object.
(So for instance $R$ could be $\mathbb{R}$ in the category of smooth manifolds.)

In \ref{sub:transport_along_a_paths} we describe how to transport elements of an algebroid along an infinitesimal variation of A-paths.
Our definition of A-path in \ref{def:A-path} is the direct translation of the classical idea found in 1.1 of \cite{MR1973056}.
By contrast to define an A-homotopy in an involution algebroid we use the flip map $\alpha:\prolong \rightarrow T(A)$ directly.
In this way we avoid using an integral (or directly appealing to the existence of a connection) as in 1.3 of \cite{MR1973056}.
In \ref{sub:transport_along_a_homotopies} we describe how to transport elements of an algebroid along an infinitesimal variation of A-homotopies.
Here we also demonstrate that this transport is `path independent' in an analogous way to how parallel transport with respect to a flat connection is independent of the path taken.
In \ref{sub:infinitesimal_a_paths} we specialise the construction of \ref{sub:transport_along_a_paths} to `infinitesimal A-paths' which arise in the composite $alg\circ w$ above.
We show that an infinitesimal A-path in $A$ is the same as a path in a fibre of $A$ starting at zero.
In \ref{sub:infinitesimal_a_homotopies} we specialise the construction of \ref{sub:transport_along_a_homotopies} to the `infinitesimal A-homotopies' which arise in the composite $alg\circ w$.
Again we show that an infinitesimal A-homotopy in $A$ is the same as a homotopy in a fibre of $A$ which starts at zero.
We leave for future work the natural next step of taking the quotient of the A-paths by the A-homotopies and identifying the extra assumptions necessary to carry out the Weinstein groupoid construction in this context.

% \subsection{Transport along A-paths and A-homotopies} % (fold)
% \label{sec:transport_along_a_paths_and_a_homotopies}

% section transport_along_a_paths_and_a_homotopies (end)

\subsection{Transport along A-paths} % (fold)
\label{sub:transport_along_a_paths}

An A-path (as described in for instance 1.1 of \cite{MR1973056}) is a path in the total space of an algebroid for which the anchor coincides with the projection of the derivative of the source.
Intuitively we think of this condition as saying that it is possible to post-compose an element $a(x_0)$ with the element $a(x_0 +h)$ where $a$ is an A-path and $h$ infinitesimally small.

Now we describe how to transport an element $a\in A$ along a variation of A-paths (a tangent vector to the space of A-paths).
Roughly speaking the $\alpha:\prolong \rightarrow T(A)$ of an involution algebroid gives a way of `infinitesimally transporting' an element $a\in A$ along an element $w\in T(A)$ to produce an element of $T(A)$ based at $a$.
Then we use a complete curve object as described in \ref{sub:curve_objects} to extend this infinitesimal transport to a full transport along an infinitesimal variation of A-paths.
The following definition is in 1.1 of \cite{MR1973056}.

\begin{definition}[Admissible path]\label{def:A-path}
	An \emph{admissible path in $A$ (or A-path)} is an arrow $a:R \rightarrow A$ such that $\anc a = T(\pi)T(a)\partial$ where $\partial$ is as in \ref{def:complete-curve-object}.
	We write $APath$ for the object of admissible paths of $A$.
\end{definition}

\begin{definition}[Tangent to admissible paths]
	The object $APath^D$ is the subobject of $T(A)^R$ of $\phi\in T(A)^R$ such that $T(\anc)\phi = c T^2(\pi)T(\phi)\partial$.
\end{definition}

At this stage we cannot immediately define the transport of A-paths as a solution to a vector field of type $X:A \rightarrow T(A)$. 
This would require being able to extend a variation of A-paths $\phi$ to a vector field on $A$.
Classically we can extend $\phi$ to a time-dependent vector field on $A$ but in a tangent category we cannot guarantee that this extension exists.
However if we assume the existence of a line object $R$ with unit $u$ and zero $0_R$ as described in \ref{sec:representable-units-function-algebras} we can produce the vector field we want on the following pullback.

\begin{definition}[Bundle of composables]
	If $\phi\in APath^D$ then \emph{the bundle $A_{\phi}$ of arrows post-composable with $\phi$} is the pullback
	\begin{equation*}
	\begin{tikzcd}[column sep = 2.5cm]
		(A_{\phi},p \pi_0) \dar{\pi_0} \rar{\pi_1} & (R\ts{T(\pi)\phi}{\pi}A, T(\pi)) \dar{(id\times \anc, id)}\\
		(R\times R, \pi_0) \rar{((\pi_0, \pi_1\bullet_p T(\pi)\phi\pi_0),id)} & (R\ts{T(\pi)\phi}{p}T(M), \pi_0)
	\end{tikzcd}
	\end{equation*}
	in the category of differential bundles.
\end{definition}

\begin{remark}
	We regard $A_{\phi}$ as a subobject $(\pi_0, \pi_1, \pi_2):A_{\phi} \rightarrowtail R\times R\times A$.
	using this decomposition the lift is $0\times \lambda_R\times \lambda$.
\end{remark}

\begin{lemma}\label{lem:infinitesimal-compn-vector-field}
	The arrow $X: A_{\phi} \rightarrow T(A_{\phi})$ defined by 
	\[X = (\partial\pi_0, 0\pi_1, \alpha(\pi_2, \pi_1\bullet_p\phi\pi_0))\]
	is a well-defined vector field.
	\begin{proof}
	It is routine to check that $X$ is a vector field and that the expression $\alpha(\pi_2, \pi_1\bullet_p \phi\pi_0)$ is well-typed.
	Now we check $X$ has codomain $T(A_{\phi})$:
	\begin{align*}
		T(\anc) \alpha(\pi_2, \pi_1\bullet_p \phi\pi_0) &= cT(\anc)(\pi_1\bullet_p \phi\pi_0)\\
		&= 0\pi_1\bullet^T_p cT(\anc)\phi\pi_0\tag{$c$ and $T(\anc)$ linear}\\
		&= T(\pi_1\bullet_p T(\pi)\phi \pi_0)(\partial\pi_0, 0\pi_1)\tag{A-path}
	\end{align*}
	and so $X$ is well-defined.
	\end{proof}
\end{lemma}

\begin{lemma}[Linearity of infinitesimal composition]
	The vector field $X$ of \ref{lem:infinitesimal-compn-vector-field} is linear over $\partial: R \rightarrow T(R)$ in the sense of \ref{def:linear-vector-field}.
	\begin{proof}
		It is routine to check that $X$ and $\partial$ commute with the projections.
		For linearity we check
		\begin{align*}
			&T(X)(0\times \lambda_R \times \lambda )\\
			&= (T(\partial)0\pi_0, T(0)\lambda_R\pi_1, T(\alpha)(\lambda\pi_2, \lambda_R\pi_1\bullet^T_p T(\phi)0\pi_0))\\
			&= (0\partial, T(0)\lambda_R\pi_1, T(\alpha)(\lambda\pi_2, l(\pi_1\bullet_p \phi\pi_0)))\tag{$\bullet$ linear}\\
			&=(0\partial, cT(\lambda_R)0\pi_1, cT(\lambda)\alpha(\pi_2, \pi_1\bullet_p \phi\pi_0))\tag{$\alpha$ linear}\\
			&= (0\times cT(\lambda_R) \times cT(\lambda))X
		\end{align*}
		and so $X$ is linear over $\partial$.
	\end{proof}
\end{lemma}

Since $X$ is linear over $\partial$ and $\partial$ is a complete vector field with solution $id: R \rightarrow R$ we can apply \ref{def:complete-curve-object} to obtain a complete solution for $X$.
So let $a\in A$ and $\phi\in APath^D$ such that $\anc a = T(\pi)\phi0_R$.
Further let $\Psi_a$ be the solution of $X$ starting at $(0_R, u, a)$.
Now we work out the properties of $\pi_0\Psi_a$ and $\pi_1\Psi_a$.
First 
\[T(\pi_0\Psi_a)\partial = \pi_0 T(\Psi_a)\partial = \pi_0 X\Psi_a = \partial \pi_0\Psi_a\]
and $\pi_0\Psi_a 0_R = 0_R$ therefore $\pi_0\Psi_a = id$.
Second
\[T(\pi_1\Psi_a)\partial = \pi_1 T(\Psi_a)\partial = \pi_1 X\Psi_a = 0\pi_1\Psi_a\]
and $\pi_1\Psi_a 0_R = u$ so $\pi_1\Psi_a = u$.
The projection $\pi_2\Psi_a$ is the transport that we want.

\begin{definition}[Infinite composition]\label{def:infinite-composition}
	If $a\in A$ and $\phi\in APath^D$ then \emph{the infinite composite $\psi_a:R \rightarrow A$ of $\phi$ starting at $a$} is $\pi_2\Psi_a$.
	In other words $\psi_a$ is the solution to $X$ starting at $(0, u, a)$ followed by $\pi_2$.
\end{definition}

The following lemma presents the properties of $\psi_a$ that we use in the sequel.

\begin{lemma}
	If $a\in A$ and $\phi \in APath^D$ then
	\begin{align*}
		\psi_a(0_R) &= a \tag{initial condition}\\
		T(\psi_a)\partial &= \alpha(\psi_a, \phi) \tag{solution to vector field}
	\end{align*}
	where $\psi_{a}$ is as defined in \ref{def:infinite-composition}.
	\begin{proof}
		Since $\Psi_a$ is the solution to $X$ starting at $(0, u, a)$ then
		\begin{align*}
			T(\psi_a)\partial &= \pi_2T(\Psi_a)\partial\\
			&= \pi_2X\Psi_a \\
			&= \alpha(\pi_2\Psi_a, \pi_1\Psi_a\bullet_p \phi \pi_0\Psi_a)\\
			&= \alpha(\psi_a, u\bullet_p \phi)
		\end{align*}
		where the last equality uses the characterisation of $\pi_0\Psi_a$ and $\pi_1\Psi_a$ established previously.
		For the initial condition: $\psi_a 0_R = \pi_2\Psi_a 0_R = a$.
	\end{proof}
\end{lemma}

The next lemma confirms that $\psi_a$ has the anchor we would expect of a transport of $a$ along $\phi$.

\begin{lemma}
	If $a\in A$ and $\phi \in APath^D$ then $\anc\psi_a = T(\pi)\phi$.
	\begin{proof}
		We check
		\begin{align*}
			\anc \psi_a &= pT(\anc)T(\psi_a)\partial\tag{$\partial$ is section}\\
			&= pT(\anc)\alpha(\psi_a, \phi) \tag{soln to vector field}\\
			& = pcT(\anc)\phi \tag{inv. algd. target}\\
			&= T(\pi)\phi
		\end{align*}
		as required.
	\end{proof}
\end{lemma}

% subsection transport_along_a_paths (end)

\subsection{Transport along A-homotopies} % (fold)
\label{sub:transport_along_a_homotopies}

Next we develop the theory analogous to that developed in \ref{sub:transport_along_a_paths} with A-homotopies in place of A-paths.
In \ref{def:admissible-homotopy} we define an A-homotopy as a map from $R\times R$ into pairs of vectors in an involution algebroid that is both horizontally and vertically an $A$-path and moreover satisfies a commutativity condition that is analogous to the classical condition determining a flat connection on a manifold.
% In classical Lie theory (for instance in \cite{MR1973056}) an A-homotopy in an integrable Lie algebroid $A$ corresponds to a map $\nabla(I\times I) \rightarrow \mathbb{G}$ into the groupoid $\mathbb{G}$ integrating $A$.
% We have used $\nabla (I\times I)$ to denote the pair (or chaotic, or indiscrete) groupoid on $I\times I$.
% In a future paper we show that the commutativity condition in \ref{def:admissible-homotopy} precisely corresponds to the fact that $\nabla(I\times I)$ is commutative.
As before transport along A-homotopies requires a complete curve object $R$ as described in \ref{sub:curve_objects}.

Therefore we show that if we start with an infinitesimal variation $h$ of A-homotopies and an element $a\in A$ such that $\anc a$ coincides with the beginning of the A-homotopy then we can transport $a$ along $h$ to obtain a map $\chi:R\times R \rightarrow A$.
In the case of A-homotopies it seems like we have a choice of how to obtain this $\chi$.
We could either integrate the horizontal A-path first and the vertical A-path second or the other way around.
In \ref{prop:path-independence} we show that our construction is independent of the choice of path we use.

\begin{definition}[Admissible homotopy]\label{def:admissible-homotopy}
	An \emph{admissible homotopy in $A$ (or A-homotopy)} is an arrow $h:R\times R \rightarrow A_2\cong A\ts{\pi}{\pi}A$ such that
	\begin{align*}
		\anc h_0 &= T(\pi)T(h_0)(\partial\times 0) \tag{hor. A-path}\\
		\anc h_1 &= T(\pi)T(h_1)(0\times \partial) \tag{vert. A-path}\\
		\alpha(h_0, T(h_1)(\partial\times 0)) & = T(h_0)(0\times \partial) \tag{A-homotopy ctd.}
	\end{align*}
	where we have used $h_i$ to denote $\pi_i h$ for $i\in \{0, 1\}$.
\end{definition}

\begin{definition}[Tangent to admissible homotopies]
	The object $AHtpy^D$ is the subobject of $T(A_2)^{R\times R}$ of $h\in T(A_2)^{R\times R}$ such that
	\begin{align*}
		T(\anc)h_0 &= cT^2(\pi)T(h_0)(\partial\times 0)\tag{hor. A-path}\\
		T(\anc)h_1 &= cT^2(\pi)T(h_1)(0 \times \partial)\tag{vert. A-path}\\
		T(\alpha)(h_0, cT(h_1)(\partial\times 0)) &= cT(h_0)(0\times \partial)\tag{A-homotopy ctd.}
	\end{align*}
	where we regard $h_0\in T(A)^{R\times R}$ and $h_1\in T(A)^{R\times R}$ where $h_i = \pi_ih$.
\end{definition}

Next we use \ref{def:infinite-composition} to define the various directions in which we can transport an element $a\in A$ given an A-homotopy $h$.

\begin{definition}[Horizontal and vertical infinite composition]\label{not:hor-and-vert-infinite-composition}
	If $h$ is an A-homotopy and $a\in A$ such that $\anc a = T(\pi)h_0(0_R, 0_R)$ then the arrows $\psi_0$, $\psi_1$, $\Phi_0$ and $\Phi_1$ are the unique infinite composites satisfying the following conditions:-
	\begin{itemize}
		\item $\psi_0:R \rightarrow A$, $\psi_0 0_R = a$ and $T(\psi_0)\partial = \alpha(\psi_0, h_0(id, 0_R))$
		\item $\psi_1: R \rightarrow A$, $\psi_1 0_R = a$ and $T(\psi_1)\partial = \alpha(\psi_1, h_1(0_R, id))$
		\item $\Phi_0:R\times R \rightarrow A$, $\Phi_0 (0_R, id) = \psi_1$ and $T(\Phi_0)(\partial\times 0) = \alpha(\Phi_0, h_0)$
		\item $\Phi_1:R\times R \rightarrow A$, $\Phi_1 (id, 0_R) = \psi_0$ and $T(\Phi_1)(0\times\partial) = \alpha(\Phi_1, h_1)$
	\end{itemize}
	where in the latter two cases we create solutions $R \rightarrow A^R$ and apply the hom-tensor adjunction.
\end{definition}

The following lemma serves to introduce the proof technique that we use repeatedly in this section and also is useful in its own right.

\begin{lemma}\label{lem:Phi_0-is-psi_0}
	\begin{equation*}
	\Phi_0(id, 0_R) = \psi_0
	\end{equation*}
	\begin{proof}
		We use the uniqueness condition in the definition of curve object.
		First we check that both sides have the same initial condition:
		\begin{equation*}
		\Phi_0(id, 0_R)0_R = \Phi_0(0_R, 0_R) = \psi_1 0_R = a = \psi_0 0_R
		\end{equation*}
		Second we check that both sides solve the same vector field.
		On the one hand:
		\begin{align*}
			T(\Phi_0(id, 0_R))\partial &= T(\Phi_0)(id, T(0_R))\partial\\
			&= T(\Phi_0)(\partial\times 0)(id, 0_R)\\
			&=\alpha(\Phi_0(id, 0_R), h_0(id, 0_R))\tag{$\Phi_0$ soln.}
		\end{align*}
		and on the other hand:
		\begin{align*}
			T(\psi_0)\partial &= \alpha(\psi_0, h_0(id, 0_R)) \tag{$\psi_0$ soln.}
		\end{align*}
		as required.
	\end{proof}
\end{lemma}

Now we show that the transport along A-homotopies is independent of the path chosen.
In other words we show $\Phi_0 = \Phi_1$.
Recall that $\Phi_0$ transports $a$ vertically then horizontally.
First we measure the vertical infinitesimal variation of $\Phi_0$.

\begin{lemma}\label{lem:hor-inf-rectangle}
	\begin{equation*}
	T(\Phi_0)(0\times \partial) = \alpha(\Phi_0, h_1)
	\end{equation*}
	\begin{proof}
		We use the uniqueness condition in the definition of curve object.
		First we check that both sides have the same initial condition:
		\begin{align*}
			\alpha(\Phi_0, h_1)(0_R, id) &= \alpha(\Phi_0(0_R, id), h_1(0_R, id))\\
			&=\alpha(\psi_1, h_1(0_R, id))\tag{init. ctd. $\Phi_0$}\\
			&=T(\psi_1)\partial \tag{$\psi_1$ soln.}\\
			&=T(\Phi_0(0_R, id))\partial \tag{init. ctd. $\Phi_0$}\\
			&= T(\Phi_0)(T(0_R)\partial, \partial)\\
			&=T(\Phi_0)(0\times\partial)(0_R, id)
		\end{align*}
		Second we check that both sides solve the same vector field.
		On the one hand:
		\begin{align*}
			T(T(\Phi_0)(0\times\partial))(\partial\times 0) &= cT(T(\Phi_0)(\partial\times0))(0\times\partial)\\
			&= cT(\alpha(\Phi_0, h_0))(0\times \partial)\tag{$\Phi_0$ is soln.}\\
			&= cT(\alpha)(T(\Phi_0)(0\times \partial), T(h_0)(0\times\partial))\\
		\end{align*}
		and on the other hand:
		\begin{align*}
			T(\alpha(\Phi_0, h_1))(\partial\times 0) &= T(\alpha)(T(\Phi_0)(\partial\times 0), T(h_1)(\partial\times 0))\\
			&= T(\alpha)(\alpha(\Phi_0, h_0), T(h_1)(\partial\times 0))\tag{$\Phi_0$ soln.}\\
			&=cT(\alpha)(\alpha(\Phi_0, h_1), cT(\alpha)(h_0, cT(h_1)(\partial\times 0)))\tag{inv. algd. flip}\\
			&=cT(\alpha)(\alpha(\Phi_0, h_1), T(h_0)(0\times\partial))\tag{A-homotopy ctd.}
		\end{align*}
		as required.
	\end{proof}
\end{lemma}

\begin{proposition}\label{prop:path-independence}
	\begin{equation*}
	\Phi_0 = \Phi_1
	\end{equation*}
	\begin{proof}
		We use the uniqueness condition in the definition of curve object.
		First we check that both sides have the same initial condition:
		\begin{align*}
			\Phi_1(id, 0_R) &= \psi_0 \tag{init. ctd. $\Phi_1$}\\
			&= \Phi_0(id, 0_R) \tag{\ref{lem:Phi_0-is-psi_0}}
		\end{align*}
		Second we check that both sides solve the same vector field.
		On the one hand $T(\Phi_0)(0\times\partial) = \alpha(\Phi_0, h_1)$ by \ref{lem:hor-inf-rectangle}.
		On the other hand $T(\Phi_1)(0\times\partial) = \alpha(\Phi_1, h_1)$ by \ref{not:hor-and-vert-infinite-composition}.
	\end{proof}
\end{proposition}

% subsection transport_along_a_homotopies (end)

% \subsection{Infinitesimal A-paths and A-homotopies} % (fold)
% \label{sec:infinitesimal_a_paths_and_a_homotopies}

% section infinitesimal_a_paths_and_a_homotopies (end)

\subsection{Infinitesimal A-paths} % (fold)
\label{sub:infinitesimal_a_paths}

In this section study the homotopy theory of involution algebroids when applied to paths and homotopies that are appropriately `infinitesimally close to an identity element'.
These infinitesimal paths and homotopies arise in a natural way in the method of integrating Lie algebroids described in \cite{MR1973056} and in a future paper we show how the theory presented in that paper can be translated to apply to involution algebroids in tangent categories.
In this section we treat infinitesimal A-paths and in the next infinitesimal A-homotopies.
The key results are \ref{prop:inf-paths-are-in-fibre} and \ref{prop:inf-A-homotopies-are-in-fibre} which show that any infinitesimal A-path or A-homotopy in an involution algebroid $A$ can be viewed as a path or homotopy respectively in a single fibre of $A$.
Both of these results require a complete curve object as described in \ref{sub:curve_objects}.

First we define infinitesimal A-paths and make precise what we mean by a path in a fibre of an involution algebroid.
Then we show how to use the flip $\alpha:\prolong \rightarrow T(A)$ to create an infinitesimal A-path $\vee \psi$ from a path in a fibre $\psi$.
In the other direction we combine $\alpha$ with solutions obtained from complete curve object $R$ (see \ref{sub:curve_objects}) to obtain a path $\wedge \phi$ in a fibre from an infinitesimal A-path $\phi$.
Then in \ref{prop:inf-paths-are-in-fibre} we show that $\vee$ and $\wedge$ are inverses.

\begin{definition}[Infinitesimal A-paths]
	The object $alg(wA)_1$ is the subobject of $T(A)^R$ consisting of the $\phi\in T(A)^R$ such that:-
	\begin{align*}
		p\phi &= \xi m!\tag{starts at zero}\\
		T(\pi)\phi 0_R &= 0 m \tag{source constant}\\
		T(\anc)\phi &= cT^2(\pi)T(\phi)\partial \tag{variation of A-paths}
	\end{align*}
	where $m = \pi p \phi 0_R$.
\end{definition}

\begin{definition}[Paths in fibres]
	The object $(A^R)^{\pi}_M$ is the subobject of $A^R$ consisting of the $\chi\in A^R$ such that:-
	\begin{align*}
		\chi 0_R &= \xi m \tag{starts at zero}\\
		\pi\chi &= m! \tag{$\pi$ constant}
	\end{align*}
	where $m = \pi\chi 0_R$.
\end{definition}

\begin{lemma}[Differentiation to A-path]
	The arrow $\vee:(A^R)^{\pi}_M \rightarrow alg(wA)_1$ defined by $\chi\mapsto \alpha(\xi m!, T(\chi)\partial)$ is well-typed.
	\begin{proof}
		First the base $m$ is preserved:
		\begin{align*}
			\pi p (\vee\chi) 0_R &= \pi p\alpha(\xi m!, T(\chi)\partial)0_R\\
			&= \pi\xi\pi\chi 0_R\tag{inv. algd. 0}\\
			&= \pi\chi 0_R = m
		\end{align*}
		Second $\vee \chi$ starts at zero:
		\begin{align*}
			p\alpha(\xi m!, T(\chi)\partial) = \xi m!  \tag{inv. algd. 0}
		\end{align*}
		Third $\vee\chi$ is source constant:
		\begin{align*}
			T(\pi)\alpha(\xi m!, T(\chi)\partial)0_R &= \anc pT(\chi)\partial 0_R \tag{inv. algd. source}\\
			&= \anc \chi0_R \tag{$\partial$ section}\\
			&= \anc \xi m \tag{starts at zero}\\
			&= 0m
		\end{align*}
		Fourth $\vee\chi$ is a variation of A-paths:
		\begin{align*}
			cT^2(\pi)T(\alpha(\xi m!, T(\chi)\partial))\partial &= cT(T(\pi)\alpha(\xi m!, T(\chi)\partial))\partial\\
			&= cT(\anc p T(\chi)\partial)\partial\tag{inv. algd. source}\\
			&= cT(\anc)T(\chi)\partial\tag{$\partial$ section}\\
			&= T(\anc)\alpha(\xi m, T(\chi)\partial) \tag{inv. algd. target}
		\end{align*}
		Therefore $\vee\chi\in alg(wA)_1$.
	\end{proof}
\end{lemma}

\begin{lemma}[Integration to path in fibre]
	The arrow $\wedge: alg(wA)_1 \rightarrow (A^R)^{\pi}_M$ defined as the unique infinite composite satisfying $(\wedge \phi) 0_R = \xi m$ and $T(\wedge \phi) = \alpha(\wedge \phi, \phi)$ is well-typed.
	\begin{proof}
		First the base point $m$ is preserved:
		\begin{equation*}
			\pi (\wedge \phi) 0_R = \pi\xi m = m
		\end{equation*}
		Second $\wedge\phi$ is $\pi$ constant:
		\begin{align*}
			\pi(\wedge\phi) &= pT(\pi)T(\wedge\phi)\partial\tag{$\partial$ section}\\
			&= pT(\pi)\alpha(\wedge\phi, \phi) \tag{$\wedge\phi$ soln.}\\
			&= p\anc p \phi \tag{inv. algd. source}\\
			&= \pi p \phi \\
			&= \pi\xi m!\tag{starts at zero}\\
			&= m!
		\end{align*}
		Third $\wedge\phi$ starts at zero by the initial condition defining $\wedge\phi$.
		Therefore $\wedge\phi\in (A^R)^{\pi}_M$.
	\end{proof}
\end{lemma}

\begin{proposition}\label{prop:inf-paths-are-in-fibre}
	The arrows $\vee$ and $\wedge$ are inverses.
	\begin{proof}
		First
		\begin{align*}
			\vee(\wedge \phi) &= \alpha(\xi m!, T(\wedge\phi)\partial)\\
			&= \alpha(\xi m!, \alpha(\wedge\phi, \phi))\tag{$\wedge\phi$ is soln.}\\
			&= \phi \tag{inv. algd. inv.}
		\end{align*}
		because $p\phi = \xi m!$.
		Second $\wedge(\vee\chi)$ is the unique infinite composite satisfying $T(\wedge\vee\chi)\partial = \alpha(\wedge\vee\chi, \vee\chi)$ and $(\wedge\vee\chi)0_R = \xi m$.
		But $\chi$ satisfies these equations.
		Indeed $\chi 0_R = \xi m$ because $\chi$ starts at zero and
		\begin{align*}
			\alpha(\chi, \vee\chi) &= \alpha(\chi, \alpha(\xi m!, T(\chi)\partial))\\
			&= T(\chi)\partial \tag{inv. algd. inv.}
		\end{align*}
		Therefore $\vee$ and $\wedge$ are inverses.
	\end{proof}
\end{proposition}

% subsection infinitesimal_a_paths (end)

\subsection{Infinitesimal A-homotopies} % (fold)
\label{sub:infinitesimal_a_homotopies}

Now we prove the analogous result to \ref{prop:inf-paths-are-in-fibre} with A-homotopies in place of A-paths.
So roughly speaking we show that an infinitesimal A-homotopy in an involution algebroid $A$ is the same as a homotopy in a fibre of $A$.
First we define infinitesimal A-homotopies in the appropriate way and make precise what we mean by a homotopy in a fibre of $A$.
Then we show how to use the flip $\alpha:\prolong \rightarrow T(A)$ to create an infinitesimal A-homotopy $\vee \chi$ from a homotopy in a fibre $\chi$.
In the other direction we combine $\alpha$ with solutions obtained from complete curve object $R$ (see \ref{sub:curve_objects}) to obtain a homotopy $\wedge h$ in a fibre from an infinitesimal A-homotopy $h$.
Then in \ref{prop:inf-A-homotopies-are-in-fibre} we show that $\vee$ and $\wedge$ are inverses.

\begin{definition}
	The object $alg(wA)_2$ is the subobject of $T(A_2)^{R\times R}$ on the elements $h\in T(A_2)^{R\times R}$ such that:-
	\begin{align*}
		ph_i &= \xi m!\tag{starts at zero}\\
		T(\pi)h_i(0_R, 0_R) &= 0m \tag{source constant}\\
		T(\anc) h_0 &= cT^2(\pi)T(h_0)(\partial \times 0)\tag{hor. A-path}\\
		T(\anc) h_1 &= cT^2(\pi)T(h_1)(0\times \partial) \tag{vert. A-path}\\
		cT(h_0)(0\times \partial) &= T(\alpha)(h_0, cT(h_1)(\partial\times 0)) \tag{A-homotopy ctd.}
	\end{align*}
	where $m = \pi\pi_0 p h(0_R, 0_R)$ and $h_i = T(\pi_i) h:R\times R \rightarrow T(A)$ for $i\in \{0, 1\}$.
\end{definition}

\begin{definition}
	The object $(A^{R\times R})^{\pi}_M$ is the subobject of $A^{R\times R}$ on the elements $\eta\in A^{R\times R}$ such that:-
	\begin{align*}
		\eta(0_R, 0_R) &= \xi m \tag{starts at zero}\\
		\pi\eta &= m! \tag{$\pi$ constant}
	\end{align*}
	where $m = \pi\eta(0_R, 0_R)$.
\end{definition}

\begin{definition}[Differentiation to admissible homotopies]
	The arrow $\vee:(A^{R\times R})^{\pi}_M \rightarrow alg(wA)_2$ defined by
	\begin{equation*}
	\eta \mapsto (\alpha(\xi m!, T(\eta)(\partial\times 0)), \alpha(\xi m!, T(\eta)(0\times \partial)))
	\end{equation*}
	is well-typed.
	\begin{proof}
		First the base point $m$ is preserved:
		\begin{align*}
			\pi\pi_0 p (\vee\eta) (0_R, 0_R) & = \pi p \alpha(\xi m!, T(\eta)(\partial\times 0)) (0_R, 0_R)\\
			&=m \tag{inv. algd. 0}
		\end{align*}
		Second $\vee\eta$ starts at zero:
		\begin{align*}
			p(\vee\eta)_0 &= p\alpha(\xi m!, T(\eta)(\partial\times 0))\\
			&=\xi m \tag{inv. algd. 0}
		\end{align*}
		Third $\vee\eta$ is source constant:
		\begin{align*}
			T(\pi)(\vee\eta)_0(0_R, 0_R) &= T(\pi)\alpha(\xi m, T(\eta)(\partial\times 0))(0_R, 0_R)\\
			&=\anc p T(\eta)(\partial\times 0)(0_R, 0_R) \tag{inv. algd. source}\\
			&= \anc \eta(0_R, 0_R) \tag{$\partial$ section}\\
			&= \anc \xi m \tag{starts at zero}\\
			&= 0m
		\end{align*}
		Fourth $(\vee\eta)_0$ is a horizontal A-path:
		\begin{align*}
			c&T^2(\pi)T((\vee\eta)_0)(\partial\times 0)\\
			&=cT^2(\pi)T(\alpha(\xi m!, T(\eta)(\partial\times 0)))(\partial\times 0)\\
			&= cT(T(\pi)\alpha(\xi m!, T(\eta)(\partial\times 0)))(\partial\times 0)\\
			&= cT(\anc p T(\eta)(\partial\times 0))(\partial\times 0) \tag{inv. algd. source}\\
			&= cT(\anc \eta)(\partial\times 0) \tag{$\partial$ section}\\
			&= T(\anc)\alpha(\xi m!, T(\eta)(\partial\times 0 )) \tag{inv. algd. target}
		\end{align*}
		and similarly for the identities involving $(\vee\eta)_1$.
		Fifth $\vee\eta$ satisfies the A-homotopy condition:
		\begin{align*}
			T&(\alpha)((\vee\eta)_0, cT((\vee\eta)_1)(\partial\times 0))\\
			&=T(\alpha)(\alpha(\xi m!, T(\eta)(\partial\times 0)), cT(\alpha(\xi m!, T(\eta)(0\times \partial)))(\partial\times 0))\\
			&=T(\alpha)(\alpha(\xi m!, T(\eta)(\partial\times 0)), cT(\alpha)(T(\xi)T(m)T(!), T^2(\eta)T(0\times \partial))(\partial\times 0))\\
			&=T(\alpha)(\alpha(\xi m !, T(\eta)(\partial\times 0)), cT(\alpha)(T(\xi)T(m)T(!), T^2(\eta)T(0\times\partial)(\partial\times 0)))\\
			&=c T(\alpha)(\alpha(\xi m!, T(\xi)T(m)T(!)), c T^2(\eta)T(0\times \partial)(\partial\times 0))\tag{inv. algd. flip}\\
			&=c T(\alpha)(\alpha(\xi m!, T(\xi)T(m)T(!)), T^2(\eta)T(\partial\times 0)(0\times \partial))\\
			&=c T(\alpha)(\alpha(\xi m!, T(\xi)T(m)T(!)), T^2(\eta)T(\partial\times 0))(0\times \partial)\\
			&=c T(\alpha)(\alpha(\xi m!, 0\xi m!), T^2(\eta)T(\partial\times 0))(0\times \partial)\\
			&= cT(\alpha(\xi m, T(\eta)(\partial\times 0)))(0\times \partial)\tag{$\alpha$ linear}\\
			&= cT((\vee\eta)_0)(0\times\partial)
		\end{align*}
		Therefore $\vee\eta\in alg(wA)_2$.
	\end{proof}
\end{definition}

\begin{definition}[Integration to homotopy in fibre]
	The arrow $\wedge:alg(wA)_2 \rightarrow (A^{R\times R})^{\pi}_M$ defined by $h \mapsto \Phi_0 = \Phi_1$ of \ref{not:hor-and-vert-infinite-composition} is well-typed.
	\begin{proof}
		As usual $\wedge$ preserves the base point $m$:
		\begin{align*}
			\pi\Phi_0(0_R,0_R) &= \pi\psi_1 0_R \tag{init. ctd. $\Phi_0$}\\
			&= \pi\xi m \tag{init. ctd. $\psi_1$}\\
			&= m
		\end{align*}
		First $\Phi_0$ starts at zero:
		\begin{align*}
			\Phi_0(0_R, 0_R) &= \psi_1 0_R \tag{init. ctd. $\Phi_0$}\\
			&= \xi m \tag{init. ctd. $\psi_1$}
		\end{align*}
		Second $\Phi_0$ is $\pi$ constant:
		\begin{align*}
			\pi \Phi_0 &= pT(\pi)T(\Phi_0)(\partial \times 0) \tag{$\partial$ section}\\
			&= p T(\pi) \alpha (\Phi_0, h_0) \tag{$\Phi_0$ soln.}\\
			&= p\anc p h_0 \tag{inv. algd. source}\\
			&= p\anc \xi m! = m! \tag{$h$ starts at zero}
		\end{align*}
		Therefore $\wedge h\in (A^{R\times R})^{\pi}_M$.
	\end{proof}
\end{definition}

\begin{proposition}\label{prop:inf-A-homotopies-are-in-fibre}
	The arrows $\vee$ and $\wedge$ are inverses.
	\begin{proof}
		On the one hand note that
		\begin{equation*}
		\vee\wedge h = (\alpha(\xi m!, T(\wedge h)(\partial \times 0)), \alpha(\xi m!, T(\wedge h)(0\times \partial)))
		\end{equation*}
		and that:
		\begin{align*}
			\alpha(\xi m!, T(\wedge h)(\partial\times 0)) &= \alpha(\xi m, T(\Phi_0)(\partial\times 0))\\
			&= \alpha(\xi m!, \alpha(\Phi_0, h_0))\tag{$\Phi_0$ soln.}\\
			&= h_0 \tag{inv. algd. inv.}
		\end{align*}
		because $ph_0 = \xi m!$. 
		Similarly $(\vee\wedge h)_1 = h_1$ and therefore $\vee\wedge h = h$.
		On the other hand $\wedge \vee \eta$ is the unique infinite composite satisfying
		\begin{align*}
			T(\wedge\vee\eta)(\partial\times 0) &= \alpha(\wedge \vee \eta, (\vee\eta)_0) \\
			(\wedge\vee \eta)(0_R, id) &= \psi_1
		\end{align*}
		which using the definition of $\psi_1$ is equivalent to satisfying the following three conditions:-
		\begin{align*}
			T(\wedge\vee\eta)(\partial\times 0) &= \alpha(\wedge \vee \eta, (\vee\eta)_0) \\
			T(\wedge\vee\eta)(0\times \partial) &= \alpha(\wedge \vee \eta, (\vee\eta)_1) \\
			(\wedge\vee\eta)(0_R, 0_R) &= \xi m
		\end{align*}
		But $\eta$ satisfies these conditions.
		Indeed $\eta(0_R, 0_R) = \xi m$ because $\eta$ starts at zero and
		\begin{align*}
			\alpha(\eta, (\vee\eta)_0) &= \alpha(\eta, \alpha(\xi m!, T(\eta)(\partial\times 0)))\\
			&= T(\eta)(\partial\times 0)\tag{inv. algd. inv.}
		\end{align*}
		and similarly for the remaining condition involving $0\times \partial$.
		Therefore $\wedge\vee\eta = \eta$ and $\wedge$ and $\vee$ are inverses.
	\end{proof}
\end{proposition}

% subsection infinitesimal_a_homotopies (end)

% section the_homotopy_theory_of_involution_algebroids (end)

\bibliography{references}{}
\bibliographystyle{plain}

\end{document}